\documentclass[12pt,a4paper]{amsart} % 
\usepackage[
  a4paper,
  top=2.8cm,
  bottom=3.5cm,
  left=2.2cm,
  right=2.2cm,
  heightrounded
]{geometry}
\usepackage{amsmath,amsfonts,amssymb,amsthm}
\usepackage{bbm}
\usepackage[colorlinks=true,linkcolor=blue,citecolor=blue,urlcolor=blue]{hyperref}
\usepackage[skip=0.8em plus 0.2em minus 0.1em]{parskip}
\usepackage[
backend=bibtex,
style=numeric,
doi=false,
url=false,
isbn=false,
eprint=true]{biblatex} 

\allowdisplaybreaks

\theoremstyle{plain}
\newtheorem{definition}{Definition}
\newtheorem{theorem}{Theorem}
\newtheorem{proposition}{Proposition}
\newtheorem{corollary}{Corollary}
\newtheorem{lemma}{Lemma} 

\newtheorem{theoremx}{Theorem}

\newtheorem{definitionx}{Definition}

\newtheorem{lemmax}{Lemma}

\theoremstyle{remark}
\newtheorem{remark}{Remark}

\title{Critical and Asymmetric Fourier Uniqueness Pairs}
\author{Torgeir Keun Lysen}
\email{torgeir.lysen@gmail.com}

\begin{document}
\begin{abstract}
    Motivated by the recent work of Kulikov, Nazarov, and  Sodin, we construct sufficient conditions for discrete subsets of $\mathbb{R}$, which lie between the supercritical and subcritical cases, to constitute Fourier uniqueness pairs. This family of critical uniqueness pairs includes pairs that are strongly asymmetric, stretching beyond those associated with zeros of zeta and L-functions, discovered by Bondarenko, Radchenko, and Seip, and getting arbitrarily close to the classical Shannon--Whittaker uniqueness pair. We also identify a somewhat more restrictive family of strongly asymmetric uniqueness pairs that yield frames and hence Fourier interpolation.
\end{abstract}
\maketitle
\begingroup
\footnotesize
\setlength{\parskip}{0pt}
\tableofcontents
\endgroup
\section{Introduction} 
We use the following definition of the Fourier transform for any $f\in L^1\cap L^2$:
\[
\hat{f}(\xi):=\int_\mathbb{R}f(x)e^{-2\pi ix\xi}\,dx.
\]
The classical Shannon--Whittaker interpolation formula states that if $\operatorname{supp}\hat{f}\subseteq [-\frac{w}{2},\frac{w}{2}]$, then 
\[
f(x)=\sum_{n\in\mathbb{Z}}f(n/w)\operatorname{sinc}(wx-n),
\]
where $\operatorname{sinc}(x)=\frac{\sin(\pi x)}{\pi x}$ (see \cite{Whittaker_1915}, \cite{1697831}). In particular, this means that if $f$ vanishes on the set $\Lambda=\frac{1}{w}\mathbb{Z}$ and $\hat{f}$ vanishes on the set $M=\mathbb{R}\backslash[-\frac{w}{2},\frac{w}{2}]$, then $f\equiv0$. We also note that by the classical Paley--Wiener theorem and the Beurling--Landau density condition (see \cite{MR102702}, \cite{Beurling1989Balayage}, \cite{Landau1967DensityConditions}),
\[
w<\; \liminf_{R \to \infty} \;\inf_{x \in \mathbb{R}}
   \frac{|\Lambda \cap [x,x+R]|}{R},
\]
this uniqueness property holds for any pair which is asymptotically denser and $\Lambda$ is uniformly discrete. Unfortunately, it only applies to functions whose Fourier transform is supported on a set of finite measure.

The study of modern Fourier uniqueness started when Radchenko and Viazovska~\cite{Radchenko2019} showed that every even Schwartz function $f$ could be expressed as
\[
f(x)=\sum_{n=0}^\infty a_n(x)f(\sqrt{n})+\sum_{n=0}^\infty \hat{a}_n(x)\hat{f}(\sqrt{n}),
\]
for some family of Schwartz functions $\{a_n(x)\}_{n=0}^\infty\subset \mathcal{S}$. Similarly to the Shannon--Whittaker pair, the existence of such an interpolation formula means that if $f$ vanishes on $\Lambda=\sqrt{\mathbb{N}_0}$ and $\hat{f}$ vanishes on $M=\sqrt{\mathbb{N}_0}$, then $f\equiv0$. However, unlike Shannon--Whittaker, we only require that $f$ and $\hat{f}$ vanish on discrete sets of points.

We will refer to pairs of sets with this uniqueness property as Fourier uniqueness pairs.
\begin{definitionx}[Fourier uniqueness pair]\label{def:fourier-uniqueness-pair}
    We call a pair $(\Lambda, M)$ where $\Lambda,M\subset \mathbb{R}$ a \emph{Fourier uniqueness pair} for some space $X\subset L^1\cap L^2$ if 
    \[
    f|_\Lambda=0\text{, }\hat{f}|_M=0\text{ and }f\in X\implies f(x)\equiv 0.
    \]
    A pair that is not a \emph{Fourier uniqueness pair} is called a \emph{Fourier non-uniqueness pair}.
\end{definitionx}
The previously mentioned pairs also have an interpolation property, which is stronger than just being a uniqueness pair.
\begin{definitionx}[Fourier interpolation pair]\label{def:fourier-interpolation-pair}
    We call a pair of multisets $(\Lambda, M)$ where $\Lambda,M\subset \mathbb{R}$ a \emph{Fourier interpolation pair} for some space $X\subset L^1\cap L^2$ if there exist families $\{a_\lambda\}_{\lambda\in\Lambda},\{b_\mu\}_{\mu\in M}\subset X$, $n:\Lambda\to\mathbb{N}_0$ and $k: M\to\mathbb{N}_0$ such that for every $f\in X$, we have
    \[
    f(x)=\sum_{\lambda\in\Lambda}a_\lambda(x) f^{(n(\lambda))}(\lambda)+\sum_{\mu\in M}b_\mu(x) \hat{f}^{(k(\mu))}(\mu),
    \]
    where the right-hand side converges pointwise. 
   
\end{definitionx}
One can view Shannon--Whittaker and Radchenko--Viazovska as the two endpoints of the same spectrum, where Radchenko--Viazovska is completely symmetric and Shannon--Whittaker is completely asymmetric.

\section{Prior results}
In 2022, Ramos and Sousa showed that there exists a sufficient density condition for discrete Fourier uniqueness~\cite{ramos2022fourier}. This condition was later significantly improved by Kulikov, Nazarov, and Sodin in 2025~\cite{kulikov2025}.
\begin{definitionx}[Supercritical and subcritical pairs]\label{def:density-criticality} Let the sequences \(\Lambda=(\lambda_j)_{j\in\mathbb{Z}}\) and \(M=(\mu_j)_{j\in\mathbb{Z}}\) be ordered so that
\begin{subequations} \label{eq:sequence-ordering}
\begin{alignat}{3}
  \dots\;&<\; \lambda_{j-1} \;&<\; \lambda_j \;&<\; \lambda_{j+1} \;<\; \dots,
  &\qquad \lim_{j\to\pm\infty} \lambda_j &= \pm\infty,
  \label{eq:lambda-ordering} \\
  \dots \;&<\; \mu_{j-1}    \;&<\; \mu_j    \;&<\; \mu_{j+1}    \;<\; \dots,
  &\qquad \lim_{j\to\pm\infty} \mu_j    &= \pm\infty.
  \label{eq:mu-ordering}
\end{alignat}
\end{subequations}

Given \(1<p,q<\infty\) with \(\frac{1}{p}+\frac{1}{q}=1\), we call a pair \((\Lambda, M)\) \emph{supercritical} if
\begin{align*}
        \limsup_{j\in\mathbb{Z}}|\lambda_j|^{p-1}(\lambda_{j+1}-\lambda_j)\leq\overline{\alpha}\quad\text{and}\quad\limsup_{j\in\mathbb{Z}}|\mu_j|^{q-1}(\mu_{j+1}-\mu_j)\leq \overline{\beta}
\end{align*}
with \(\overline{\alpha}^\frac{1}{p}\overline{\beta}^\frac{1}{q}< \frac{1}{2}\). Conversely, we refer to the pair as \emph{subcritical} if 
\begin{align*}
        \liminf_{j\in\mathbb{Z}}|\lambda_j|^{p-1}(\lambda_{j+1}-\lambda_j)\geq\underline{\alpha}\quad\text{and}\quad\liminf_{j\in\mathbb{Z}}|\mu_j|^{q-1}(\mu_{j+1}-\mu_j)\geq \underline{\beta}
\end{align*}
with 
\(\underline{\alpha}^\frac{1}{p}\underline{\beta}^\frac{1}{q}> \frac{1}{2}\).
\end{definitionx}
\begin{theoremx}[Kulikov--Nazarov--Sodin]\label{thm:kns-uniqueness}
    \textit{Suppose that $1 < p, q < \infty$, $\tfrac{1}{p} + \tfrac{1}{q} = 1$. Then}
\begin{itemize}
    \item[(i)] \textit{any supercritical pair $(\Lambda, M)$ is a uniqueness pair for $\mathcal{H}$;}
    \item[(ii)] \textit{any subcritical pair $(\Lambda, M)$ is a non-uniqueness pair for $\mathcal{H}$.}
\end{itemize}
\end{theoremx}
Here $\mathcal{H}$ is the Fourier symmetric Sobolev space
\[
\mathcal{H}:=\left\{f\in L^2: \int_\mathbb{R}(1+x^2)|f(x)|^2\,dx+\int_\mathbb{R}(1+\xi^2)|\hat{f}(\xi)|^2\,d\xi<\infty\right\}.
\]
The most natural example of supercritical pairs is 
\[
\Lambda=\left\{\pm(\alpha pj)^\frac{1}{p}\right\}_{j\in\mathbb{N}}\quad\text{and}\quad M=\left\{\pm(\beta qj)^\frac{1}{q}\right\}_{j\in\mathbb{N}},
\]
where $\alpha^\frac{1}{p}\beta^\frac{1}{q}<\frac{1}{2}$ and $\frac{1}{p}+\frac{1}{q}=1$.
Similar conditions have also been explored by Ramos and Sousa~\cite{ramos2024pauli}, and in $\mathbb{R}^d$ by Adve~\cite{adve2023density}.

\subsection{The critical density gap}
The case of \(\overline{\alpha}^\frac{1}{p}\overline{\beta}^\frac{1}{q},\underline{\alpha}^\frac{1}{p}\underline{\beta}^\frac{1}{q}=\frac{1}{2}\) remains largely unexplored. We will refer to such pairs as critical density pairs; note that this is equivalent to $\overline{\alpha}, \underline{\alpha}, \overline{\beta}, \underline{\beta}=\frac{1}{2}$ because of the scaling property of the Fourier transform $\mathcal{F}\{f(ax)\}(\xi)=\frac{1}{|a|}\hat{f}\left(\frac{\xi}{a}\right)$. Clearly, the pairs $(\pm\sqrt{\mathbb{N}_0}, \pm\sqrt{\mathbb{N}_0})$ described by Radchenko and Viazovska~\cite{Radchenko2019} and, similarly, the perturbed versions seen in Ramos and Sousa~\cite{Ramos_2023} and Cassese and Ramos~\cite{cassese2025regularityfourierinterpolationformulas} are critical density pairs satisfying some critical uniqueness condition. We also have a general sufficient condition (uniform supercriticality)~\cite{kulikov2025} for such critical density pairs to be Fourier uniqueness pairs:
\begin{align*}
    \sup_{j\in \mathbb{Z}}\max\{|\lambda_j|, |\lambda_{j+1}|\}(\lambda_{j+1}-\lambda_j),\; \sup_{j\in\mathbb{Z}}\max\{|\mu_j|,|\mu_{j+1}|\}(\mu_{j+1}-\mu_j)\leq\frac{1}{2}.
\end{align*}
In contrast to general supercritical pairs, we require some density around the origin to achieve the uniqueness property. We also notice that all the previous examples correspond to the symmetric case of $p,q=2$, and to our knowledge there are no known critical Fourier uniqueness pairs for $p,q\neq 2$.

This leads to the first motivating question:
\begin{center}
\textit{Do there exist critical Fourier uniqueness pairs for $p,q\neq 2$?}
\end{center}
\begin{remark}
    A general critical uniqueness condition cannot hold without requiring some density around the origin. The uniqueness pair in \cite{Radchenko2019} is tight, meaning that removing just a single point breaks the uniqueness property.
\end{remark}
\subsection{The super-asymmetric gap}
Critical pairs correspond asymptotically to the sets $\Lambda=\left\{\pm(\frac{p}{2}j)^\frac{1}{p}\right\}_{j\in\mathbb{N}_0}$ and $M=\left\{\pm(\frac{q}{2}j)^\frac{1}{q}\right\}_{j\in\mathbb{N}_0}$. It is then also natural to wonder about the pairs where $\Lambda$ or $M$ is sparser than $\{\pm j^{1-\varepsilon}\}_{j\in\mathbb{N}_0}$ for all $0<\varepsilon<1$.

In 2023, Bondarenko, Radchenko, and Seip discovered the interpolation pair $\Lambda=\left\{\frac{\log j}{4\pi}\right\}_{j\in\mathbb{N}_0}$ and $M=\left\{\frac{\rho-\frac{1}{2}}{i}\right\}$ for even functions in $\mathcal{H}_1$, where the $\rho$ are the nontrivial zeros of the Riemann zeta function~\cite{Bondarenko2023}. Here $\mathcal{H}_1$ denotes the space of functions $f(z)$ that are analytic in the strip $\bigl|\text{Im}\, z\bigr| < \tfrac12 + \varepsilon$
and satisfy the integrability condition
\[
\sup_{|y|<\tfrac12+\varepsilon}
\int_{-\infty}^{\infty}
\bigl|f(x+iy)\bigr|\,(1+|x|)\,dx
\;<\;\infty
\]
for some $\varepsilon>0$. Assuming the Riemann hypothesis is true, we have $M\subset\mathbb{R}$ and by the Riemann--von Mangoldt formula $\frac{\rho_j-\frac{1}{2}}{i}\sim \gamma_j\sim\frac{2\pi j}{\log j}$ where $\gamma_j=\operatorname{Im}(\rho_j)$. 

In 2021, Kulikov showed that under mild conditions there exists a necessary condition for discrete Fourier interpolation pairs \cite{Kulikov2021}. For all $W,T>1$, there should exist a $C>0$ such that 
\begin{align}\label{eq:kulikov-interpolation-condition}
|\Lambda\cap[-T,T]| +|M\cap[-W,W]|\geq4WT - C \log^2(4WT).
\end{align}
Kulikov's condition combined with the existence of the Bondarenko--Radchenko--Seip pair implies the existence of a class of uniqueness pairs which are significantly more uneven or asymmetric in their asymptotic growth than critical or supercritical pairs.

This naturally raises the question:
\begin{center}
\textit{Does there exist a sufficient uniqueness condition for strongly asymmetric pairs?}
\end{center}
However, even the Bondarenko--Radchenko--Seip pair is much less asymmetric than the classical Shannon--Whittaker pair.

Hence, one may also ask:
\begin{center}
    \textit{Is there a sufficient uniqueness condition that bridges the gaps between Kulikov--Nazarov--Sodin, Bondarenko--Radchenko--Seip, and Shannon--Whittaker?}
\end{center}
We will provide an answer to all three questions. Additionally, we prove that frames exist even for very asymmetric uniqueness pairs. Our main tool will be a family of inequalities that can be viewed as fractional, global versions of the Poincar\'{e}--Wirtinger inequality.
\section{Uniform criticality}
As in the uniformly supercritical case, we can construct critical density pairs by enforcing some density condition around the origin. However, our condition will be more general.
\begin{definition}[Subexponentially admissible]\label{def:subexp-admissible}
We say that a $C^1$-function $S:[0,\infty)\to[0,\infty)$ is \emph{subexponentially admissible} if it is convex, increasing and satisfies:
\begin{itemize}
  \item[(i)] There exists $T_0>0$ such that $S\in C^3([T_0,\infty))$ and the function 
  $t\mapsto \tfrac{d}{dt}S^2(t)$ is log-concave for all $t\geq T_0$;
  \item[(ii)] The integrability condition
  \[
  \int_\mathbb{R}\frac{|\log \max\{S^2(t),1\}|}{1+t^2}\,dt<\infty
  \]
  holds.
\end{itemize}
We also define the (left-continuous) inverse of $S$ by $S^{-1}(y) := \inf\{\,x \ge 0 : S(x) \ge y\,\}$.
\end{definition}
We will in particular be interested in the cases of $S(t)=t^a$ for some $1\leq a<\infty$ and $S(t)=ae^{bt^s}$ for some $0<a,b<\infty$ and $0<s<1$.
\begin{theorem}\label{thm:uniform-critical-uniqueness}
  Let $S$ or $S^{-1}$ be subexponentially admissible. Then there exist $d,\tilde{d}\in\mathbb{R}$ such that if $(\Lambda, M)$ satisfies (\ref{eq:sequence-ordering}),
    \begin{align*}
    \sup_{j\in \mathbb{Z}}S(|\lambda_j|+\tilde{d})(\lambda_{j+1}-\lambda_j)\leq \tfrac{1}{2},\quad\text{and}\quad   \sup_{j\in\mathbb{Z}}S^{-1}(|\mu_j|+d)(\mu_{j+1}-\mu_j)\leq\tfrac{1}{2},
    \end{align*}
then $(\Lambda,M)$ is a uniqueness pair for $\mathcal{H}$.
\end{theorem}
 We can use Theorem~\ref{thm:uniform-critical-uniqueness} to create a clear analogy to uniformly supercritical pairs for general $p,q$ in the form of the following corollary.
\begin{corollary}\label{cor:critical-pairs-general-pq}
    For all $1<p,q<\infty$ and $\frac{1}{p}+\frac{1}{q}=1$, there exist $d_p,d_q\in\mathbb{R}$ such that if $(\Lambda, M)$ satisfies (\ref{eq:sequence-ordering}),
\begin{align*}
    \sup_{j\in \mathbb{Z}}(|\lambda_j|+d_p)^{p-1}(\lambda_{j+1}-\lambda_j)\leq \frac{1}{2},\quad\text{and}\quad
    \sup_{j\in\mathbb{Z}}(|\mu_j|+d_q)^{q-1}(\mu_{j+1}-\mu_j)\leq\frac{1}{2}.
\end{align*}
then $(\Lambda,M)$ is a uniqueness pair for $\mathcal{H}$.
\end{corollary}
\begin{proof}
    Without loss of generality, assume $p\leq q$. Then choose $S(t)=t^{q-1}$ and $S^{-1}(t)=t^{p-1}$. Since $S$ is subexponentially
    admissible, we get the uniqueness from Theorem~\ref{thm:uniform-critical-uniqueness}.
\end{proof}
This clearly demonstrates the existence of critical uniqueness pairs for $p,q\neq 2$. 
\begin{remark}
    We can upper bound $d_p$ and $d_q$. However, determining the optimal constants is difficult, since they are likely just an artifact of our methods. For $S(x)=x$, which corresponds to the symmetric critical case, we know that $0<d_2$ by \cite{Radchenko2019} and by \cite{kulikov2025} we know that the space is finite-dimensional for any $0<d_2$.
\end{remark}

We can also create an analogy to general supercritical pairs with Theorem~\ref{thm:uniform-critical-uniqueness}. However, we can only show that the space of functions vanishing on such a pair is finite-dimensional and not $\{0\}$.
\begin{corollary}\label{cor:finite-dim-vanishing-space}
    Let $\nu(t):[0,\infty)\to[0,\infty)$ be some function such that $\displaystyle{\lim_{t\to\infty}}\nu(t)=\infty$ and $S(t)$ be subexponentially admissible. If there exists some $N\in\mathbb{N}$ such that 
    \begin{align*}
    \sup_{|j|\geq N}S(|\lambda_j|+\nu(|\lambda_j|))(\lambda_{j+1}-\lambda_j)\leq \frac{1}{2}\quad\text{and}\quad
    \sup_{|j|\geq N}S^{-1}(|\mu_j|+\nu(|\mu_j|))(\mu_{j+1}-\mu_j)\leq\frac{1}{2},
\end{align*}
for some pair $(\Lambda,M)$, then the space of functions vanishing on $(\Lambda,M)$ is finite-dimensional.
\end{corollary}
Here, the supercritical case corresponds to $\nu: t\mapsto\varepsilon t$ for some $\varepsilon>0$ and $S(x)=x^{p-1}$ for some $1<p<\infty$, and any $\nu(t)=o(t)$ corresponds to a critical pair.

To display the full capabilities of Theorem~\ref{thm:uniform-critical-uniqueness}, we provide an example of a family of asymmetric uniqueness pairs, the most asymmetric of which almost matches the asymptotic growth of those found in \cite{Bondarenko2023}.
\begin{corollary}
\label{cor:highly-asymmetric-pairs}
Let $0<s<1$. Then there exist sequences satisfying
\[
\lambda_j \sim sign(j)\left(\frac{\log |j|}{4\pi}\right)^{1/s}
\quad\text{and}\quad \mu_j \sim \frac{j}{2} \left(\frac{4\pi}{\log |j|}\right)^{1/s}
\]
as $j\to\pm\infty$, that are Fourier uniqueness pairs for $\mathcal{H}$.
\end{corollary}
\begin{proof}
 Similarly to Corollary~\ref{cor:critical-pairs-general-pq}, we apply Theorem~\ref{thm:uniform-critical-uniqueness} with a specific choice of $S$.
 
    Let $S(t)=2\pi e^{4\pi t^s}$ and hence $S^{-1}(t)=\left(\frac{1}{4\pi}\log\frac{t}{2\pi}\right)^\frac{1}{s}$.
    Then by Theorem~\ref{thm:uniform-critical-uniqueness} there exist constants $d,\tilde{d}\geq 0$ such that any pair $(\Lambda,M)$ satisfying (\ref{eq:sequence-ordering}),
    \begin{align*}
        \sup_{j\in\mathbb{Z}}2\pi e^{4\pi (|\lambda_j|+\tilde{d})^s}(\lambda_{j+1}-\lambda_j)\leq\frac{1}{2}, \quad\text{and}\quad
        \sup_{j\in\mathbb{Z}}\left(\frac{1}{4\pi}\log\frac{|\mu_j|+d}{2\pi}\right)^\frac{1}{s}(\mu_{j+1}-\mu_j)\leq\frac{1}{2}
    \end{align*}
    is a Fourier uniqueness pair.

   Consider the sequences $\tilde{\Lambda}:=\{\tilde{\lambda}_j\}_{j\geq J}$ and $\tilde{M}:=\{\tilde{\mu}_j\}_{j\geq J}$ where 
   \begin{align*}
          \tilde{\lambda}_j:=\left(\frac{\log j}{4\pi}-\frac{1-s}{4\pi s}\log\frac{\log j}{4\pi}+\frac{\log s}{4\pi}\right)^\frac{1}{s}-\tilde{d},\quad\tilde{\mu}_j:=\frac{j}{2}\left(\frac{4\pi}{\log j}\right)^\frac{1}{s}-d,
   \end{align*}
   and $J\in\mathbb{N}$ is sufficiently large so that $\tilde{\mu}_j$ and $\tilde{\lambda}_j$ are non-negative and well-defined. This pair $(\tilde{\Lambda},\tilde{M})$ satisfies the inequalities above. 
   
   To satisfy (\ref{eq:sequence-ordering}) we can add finitely many points to the sequences $\tilde{\Lambda}\cup(-\tilde{\Lambda})$ and $\tilde{M}\cup(-\tilde{M})$. Hence, we get a uniqueness pair $(\Lambda,M)$ that satisfies the desired asymptotics.   
\end{proof}
Letting $s\to1$, we see that this asymptotically matches the pair found by Bondarenko, Radchenko, and Seip. As noted in \cite{Kulikov2021} and \cite{Bondarenko2023}, applying (\ref{eq:kulikov-interpolation-condition}) to the interpolation pair in \cite{Bondarenko2023} yields the inequality
\[
N(T)\geq\frac{T}{2\pi}\log\frac{T}{2\pi e}-C\log^2T,
\]
where $N(T):=|\{\rho:0<\operatorname{Im}(\rho)< T\}|$. This is strikingly similar to the Riemann--von Mangoldt formula
\[
N(T)=\frac{T}{2\pi}\log\frac{T}{2\pi e}+O(\log T).
\]
As we will see in Proposition~\ref{prop:kulikov-density-condition}, this near equality also applies to the sparsest uniqueness pairs described by Theorem~\ref{thm:uniform-critical-uniqueness}, meaning that they all exhibit a similar kind of critical density.
\begin{proposition}\label{prop:kulikov-density-condition}
    Let $S:[0,\infty)\to[0,\infty)$ be an unbounded, increasing $C^1$-function and $(\Lambda,M)$ be a pair satisfying (\ref{eq:sequence-ordering}),
    \begin{align*}
        \lim_{j\to\pm\infty}S(|\lambda_j|)(\lambda_{j+1}-\lambda_j)=\alpha,\quad\text{and}\quad\lim_{j\to\pm\infty}S^{-1}(|\mu_j|)(\mu_{j+1}-\mu_j)=\alpha
    \end{align*}
    for some $0<\alpha<\infty$. Then 
    \[
    \inf_{W,T\geq K}\frac{|\Lambda\cap[-T,T]|+|M\cap[-W,W]|}{4WT}\to \frac{1}{2\alpha}
    \]
    as $K\to\infty$.
\end{proposition}
Combining Proposition~\ref{prop:kulikov-density-condition} with Kulikov's condition
\begin{align*}
|\Lambda\cap[-T,T]| +|M\cap[-W,W]|\geq4WT - C \log^2(4WT),
\end{align*}
suggests that our conditions are close to optimal, since the required density for Fourier interpolation pairs and uniqueness pairs is likely comparable in $\mathcal{S}$ and $\mathcal{H}$. However, this by itself is not enough to guarantee that Theorem~\ref{thm:uniform-critical-uniqueness} is optimal as there could exist Fourier uniqueness pairs for $\mathcal{H}$ that are not interpolation pairs.
\begin{remark}
    A classical result \cite[p.~246]{titchmarsh1986theory} states that there exists an $\varepsilon>0$ such that 
\[
\limsup_{j \to \infty} \frac{\log \gamma_j}{4\pi}(\gamma_{j+1} - \gamma_j) > \frac{1}{2} + \varepsilon
\quad \text{and} \quad
\liminf_{j \to \infty} \frac{\log \gamma_j}{4\pi}(\gamma_{j+1} - \gamma_j) < \frac{1}{2} - \varepsilon,
\]
where $\gamma_j=\operatorname{Im}(\rho_j)$ is the imaginary part of the nontrivial zeros of the zeta function. Additionally, under the GUE hypothesis, the normalized spacing distribution agrees with that of GUE eigenvalues, which has unbounded support on 
$[0,\infty)$ \cite{Odlyzko1987ZetaSpacings}. Consequently, in that case,
\[
\limsup_{j \to \infty} \frac{\log \gamma_j}{4\pi}(\gamma_{j+1} - \gamma_j) =\infty
\quad \text{and} \quad
\liminf_{j \to \infty} \frac{\log \gamma_j}{4\pi}(\gamma_{j+1} - \gamma_j)=0.
\]
Choosing $S(t)=2\pi e^{4\pi t^s}$ and hence $S^{-1}(t)=\left(\frac{1}{4\pi}\log\frac{t}{2\pi}\right)^\frac{1}{s}$ as in Corollary~\ref{cor:highly-asymmetric-pairs}, but with $s=1$, this implies that the interpolation points do not need to satisfy
\begin{align*}
        \limsup_{j\to\pm\infty}S(|\lambda_j|)(\lambda_{j+1}-\lambda_j)\leq\frac{1}{2}\quad\text{and}\quad\limsup_{j\to\pm\infty}S^{-1}(|\mu_j|)(\mu_{j+1}-\mu_j)\leq\frac{1}{2}
\end{align*}
and can deviate significantly at infinitely many points as long as they still satisfy (\ref{eq:kulikov-interpolation-condition}). 
\end{remark}
\section{Pushing asymmetry}
\subsection{Uniqueness}
We can actually push the asymmetry far beyond Corollary~\ref{cor:highly-asymmetric-pairs} and the pair in Bondarenko--Radchenko--Seip, and arbitrarily close to Shannon--Whittaker. 

\begin{definition}[Widely admissible]\label{def:widely-admissible}
An increasing function \(S:[0,\infty)\to [0,\infty)\) is called \emph{widely admissible} if $S(x)=\tilde{S}(x^s)|_{[0,\infty)}$ for some $0<s<\infty$ where

\begin{enumerate}
\item[(i)] $\tilde{S}:\mathbb{C}\to\mathbb{C}$ is entire with $\tilde S^{(n)}(0)\ge 0$ for all $n\in\mathbb{N}_0$;
\item[(ii)] There exists a $0<\delta<\frac{1}{2}$ such that $\displaystyle{\limsup_{x\to\infty}\tilde{S}'(x)\tilde{S}(x)^{\delta-\frac{3}{2}}=0}$.
\end{enumerate}
We also define the (left-continuous) inverse of $S$ by $S^{-1}(y) := \inf\{\,x \ge 0 : S(x) \ge y\,\}$.
\end{definition}
We can show that this class of functions is quite large and that condition (ii) does not restrict the growth of $S$. 
\begin{proposition}\label{prop:widely-admissible-dominates}
    For any non-decreasing function $f:[0,\infty)\to[0,\infty)$ there exists a widely admissible $S$ such that $f(x)=o(S(x))$ as $x\to\infty$.
\end{proposition}
\begin{theorem}\label{thm:widely-admissible-uniqueness}
         Let $S$ be widely admissible. There exists a function $\varepsilon:\mathbb{R}\to\mathbb{R}$ with $\displaystyle{\lim_{t\to\pm\infty}}\varepsilon(t)=0$ such that if $(\Lambda, M)$ satisfies (\ref{eq:sequence-ordering}), 
\begin{align*}
    \sup_{j\in \mathbb{Z}}(1+\varepsilon(\lambda_j))\max\{S^{-1}(|\lambda_j|),1\}(\lambda_{j+1}-\lambda_j)&\leq \frac{1}{2},\quad\text{and}\\
    \sup_{j\in\mathbb{Z}}(1+\varepsilon(\mu_j))\max\{S(|\mu_j|),1\}(\mu_{j+1}-\mu_j)&\leq\frac{1}{2},
\end{align*}
then $(\Lambda,M)$ is a uniqueness pair for $\mathcal{H}$.
\end{theorem}
In contrast to previous results, we can make these uniqueness pairs arbitrarily asymmetric by Proposition~\ref{prop:widely-admissible-dominates}. 
\begin{remark}
    By a similar calculation as in the proof of Corollary~\ref{cor:highly-asymmetric-pairs}, we get uniqueness pairs of the form 
\[
\lambda_j\sim sign(j)\frac{\log|j|}{4\pi}\quad\text{and}\quad\mu_j\sim\frac{2\pi j}{\log |j|}\sim sign(j)\gamma_j
\]
as $j\to\pm\infty$. We thus have infinitely many uniqueness pairs that asymptotically match those in \cite{Bondarenko2023}. 
\end{remark}
\begin{remark}
Certain choices of $S$ are both subexponentially admissible and widely admissible. In these cases, we may apply either Theorem~\ref{thm:uniform-critical-uniqueness} or Theorem~\ref{thm:widely-admissible-uniqueness}. Theorem~\ref{thm:uniform-critical-uniqueness}, however, provides a more explicit characterization of $\varepsilon(t)$, namely $\varepsilon_1(\lambda)=\frac{S(|\lambda|+d)}{\max\{S(|\lambda|),1\}}-1$ and $\varepsilon_2(\mu)=\frac{S^{-1}(|\mu|+d)}{\max\{S^{-1}(|\mu|),1\}}-1$. Moreover, Theorem~\ref{thm:uniform-critical-uniqueness} allows us to choose $S$ under weaker smoothness assumptions than Theorem~\ref{thm:widely-admissible-uniqueness}.
\end{remark}
We can also get a result similar to Corollary~\ref{cor:finite-dim-vanishing-space}.
\begin{corollary}\label{cor:finite-dim-widely-admissible}
        Let $S$ or $S^{-1}$ be widely admissible. Then the space of functions in $\mathcal{H}$ vanishing on a pair $(\Lambda, M)$ satisfying (\ref{eq:sequence-ordering}),
\begin{align*}
    \limsup_{j\in \mathbb{Z}}S^{-1}(|\lambda_j|)(\lambda_{j+1}-\lambda_j)<\frac{1}{2},\quad\text{and}\quad
    \limsup_{j\in\mathbb{Z}}S(|\mu_j|)(\mu_{j+1}-\mu_j)<\frac{1}{2}
\end{align*}
is finite-dimensional.
\end{corollary}
If this space is not just finite-dimensional, but trivial, one would again get a result analogous to Theorem~\ref{thm:kns-uniqueness}.
\subsection{Stable uniqueness and interpolation}
We define the Hilbert space $H( S)\subset L^2$ with the norm
\begin{align*}
    \|f\|_{H(S)}^2:=\int_\mathbb{R}(1+x^2)|f(x)|^2\,dx+\int_\mathbb{R}(1+S^2(|\xi|))|\hat{f}(\xi)|^2\,d\xi.
\end{align*}
Giving up critical density, we can quantify the uniqueness in terms of the following frame bounds.
\begin{theorem}\label{thm:main-sampling-inequality}
    Let $S$ be widely admissible and $0<\underline{\alpha}\leq\overline{\alpha}<\frac{1}{2}$. Then there exist constants $K_{\overline{\alpha},S},C>0$ such that if $(\Lambda,M)$ satisfies (\ref{eq:sequence-ordering}) and
    \begin{align*}
        \underline{\alpha}\leq\max\{S^{-1}(|\lambda_j|),K_{\overline{\alpha},S}\}(\lambda_{j+1}-\lambda_j),\,\max\{S(|\mu_j|),K_{\overline{\alpha},S}\}(\mu_{j+1}-\mu_j)\leq \overline{\alpha}\quad\text{for all $j\in\mathbb{Z}$},
    \end{align*}
    then for any $f\in\mathcal{H}$, we have
    \begin{align*}
    C^{-1}\|f\|_{H(S)}^2
    &\leq \sum_{j\in\mathbb{Z}}(\lambda_{j+1}-\lambda_j)(1+\lambda_j^2)|f(\lambda_j)|^2+\sum_{j\in\mathbb{Z}}(\mu_{j+1}-\mu_j)(1+S^2(|\mu_j|))|\hat{f}(\mu_j)|^2\\
    &\leq C\|f\|_{H(S)}^2.
    \end{align*}
\end{theorem}
One can view these frame bounds as natural discretizations of the integrals defining the norms $\|f\|^2_{H( S)}$.

It is well-known that any such
frame bounds give rise to an interpolation formula, see e.g. \cite[Section 5.1]{MR1843717}, \cite[Section 1.8]{MR3468930} or \cite{kulikov2025}.
\begin{corollary}\label{cor:interpolation-widely-admissible}
Let $S$ be widely admissible and $0<\alpha<\frac{1}{2}$. Then there exists a constant $K_{\alpha,S}>0$ such that if $(\Lambda,M)$ satisfies (\ref{eq:sequence-ordering}) and
    \begin{align*}
        \sup_{j\in\mathbb{Z}}\max\{S^{-1}(|\lambda_j|),K_{\alpha,S}\}(\lambda_{j+1}-\lambda_j),\,\sup_{j\in\mathbb{Z}}\max\{S(|\mu_j|),K_{\alpha,S}\}(\mu_{j+1}-\mu_j)\leq \alpha
    \end{align*}
    then $(\Lambda,M)$ is a Fourier interpolation pair for $H(S)$. In particular, this means that for all $f\in H(S)$ there exist functions $\{a_\lambda\}_{\lambda\in\Lambda},\{b_\mu\}_{\mu\in M}\subset H(S)$ such that 
    \[
    f=\sum_{\lambda\in\Lambda}a_\lambda f(\lambda)+\sum_{\mu\in M}b_\mu \hat{f}(\mu),
    \]
    where the right-hand side converges in $H(S)$.
\end{corollary}
\begin{remark}
    Let 
    \[
    \mathfrak{S}:=\{\text{$S$ is widely admissible}\}.
    \]
    Then by Proposition~\ref{prop:widely-admissible-dominates}, we know that $f\in \displaystyle{\bigcap_{ S\in\mathfrak{S}}H(S)}$ if and only if 
    \[
    f(x)\in PW_\sigma^2\cap L^2(\mathbb{R}, (1+x^2)\,dx)
    \] for some $\sigma>0$, meaning that the limiting case of Corollary~\ref{cor:interpolation-widely-admissible} is essentially a special case of Beurling--Landau sampling in a  Paley--Wiener space $PW_\sigma^2$. 
\end{remark}

\begin{remark}
    Note that Corollary~\ref{cor:interpolation-widely-admissible} does not imply the existence of new crystalline measures that are not covered by Theorem 2 in \cite{kulikov2025}. Indeed, if $S$ grows faster than any polynomial, then $\mathcal{S}$ is not a subspace of $H(S)$. Similarly, to obtain interpolation formulas that asymptotically match those in \cite{Bondarenko2023}, we must require the function to be holomorphic in a strip around $\mathbb{R}$, since the growth of $S$ forces $\hat{f}$ to decay exponentially.
\end{remark}

We know that Corollary~\ref{cor:interpolation-widely-admissible} and consequently Theorem~\ref{thm:main-sampling-inequality} are asymptotically optimal since Proposition~\ref{prop:kulikov-density-condition} combined with Kulikov's condition (\ref{eq:kulikov-interpolation-condition}) implies that there are no interpolation pairs $(\Lambda, M)$ for the spaces $H(S)$ where $S:[0,\infty)\to[0,\infty)$ is convex and increasing that satisfy (\ref{eq:sequence-ordering}), 
\begin{align*}
    \liminf_{j\to \pm\infty}S^{-1}(|\lambda_j|)(\lambda_{j+1}-\lambda_j)> \frac{1}{2},\quad\text{and}\quad   \liminf_{j\to\pm\infty}S(|\mu_j|)(\mu_{j+1}-\mu_j)>\frac{1}{2}.
\end{align*}

\section{Proof of Proposition~\ref{prop:kulikov-density-condition}}
\begin{lemma}\label{lem:fenchel-young-integral}
     Let $S(\lambda):[0,\infty)\to[0,\infty)$ be an unbounded, increasing, continuously differentiable function with $S(0)=0$. Then
     \[
     \int_0^TS(\lambda)\,d\lambda+\int_0^WS^{-1}(\mu)\,d\mu\geq WT
     \]
     with equality if and only if $W=S(T)$.
\end{lemma}
\begin{proof}
    We just need to show that 
    \[
    g(W):=\int_0^WS^{-1}(\mu)\,d\mu\quad\text{and}\quad f(T):=\int_0^TS(\lambda)\,d\lambda
    \]
    are convex conjugates. Then the rest follows from the Fenchel--Young inequality.

    Since $f$ is strictly convex and differentiable, we can find the convex conjugate by
    \[
    f^*(y)=yx(y)-f(x(y)),
    \]
    where $f'(x(y))=y$. Since $f'(y)=S(y)$, we know that $x(y)=S^{-1}(y)$. Hence, 
    \[
    f^*(y)=yS^{-1}(y)-f(S^{-1}(y)).
    \]
    Differentiating we get
    \[
    (f^*)'(y)=y(S^{-1})'(y)+S^{-1}(y)-(S^{-1})'(y)S(S^{-1}(y))=S^{-1}(y).
    \]
    So, 
    \[
    f^*(W)=\int_0^WS^{-1}(\mu)\,d\mu+C,
    \]
    but $f^{*}(0)=0$ so $C=0$ and thus $f^*=g$.
\end{proof}
We will now use our Fenchel--Young inequality to prove Proposition~\ref{prop:kulikov-density-condition}.
\begin{proof}[Proof of Proposition~\ref{prop:kulikov-density-condition}]
    Let
    \begin{align*}
        f(T):=\sum_{j\in\mathbb{Z}}\frac{T-\min\{|\lambda_j|,|\lambda_{j+1}|\}}{\lambda_{j+1}-\lambda_j}\mathbbm{1}_{I_j}(T)+|\Lambda\cap[-T,T]|,
    \end{align*}
    where $I_j=[\min\{[|\lambda_j|,|\lambda_{j+1}|\}, \max\{[|\lambda_j|,|\lambda_{j+1}|\}]$. Since $\lambda_{j+1}-\lambda_j\to0$ as $j\to\pm\infty$, we know that for any $\varepsilon>0$, there exists a $T_0\geq 0$ such that
    \[
    f\left(\frac{T}{1+\varepsilon}\right)\leq|\Lambda\cap[-T,T]|\leq f(T(1+\varepsilon))
    \]
    for all $T\geq T_0$. We know that $f$ is continuous, and if $\lambda_m< T< \lambda_{m+1}$ and $\lambda_n<-T< \lambda_{n+1}$, then
    \[
    f'(T)=\frac{1}{\lambda_{m+1}-\lambda_m}+\frac{1}{\lambda_{n+1}-\lambda_n}.
    \]
    Since
    \[
    \alpha^{-1}S(|\lambda_j|)(\lambda_{j+1}-\lambda_j)\to1
    \]
    as $j\to\pm\infty$, we know that for any $\varepsilon>0$, we have
    \[
    (2\alpha^{-1}+o(1))\int_0^\frac{T}{1+\varepsilon}S(\lambda)\,d\lambda\leq f(T)\leq(2\alpha^{-1}+o(1))\int_0^{T(1+\varepsilon)}S(\lambda)\,d\lambda.
    \]
    Hence, for any $\varepsilon>0$, we have
    \[
    (2\alpha^{-1}+o(1))\int_0^\frac{T}{1+\varepsilon}S(\lambda)\,d\lambda\leq |\Lambda\cap[-T,T]|\leq(2\alpha^{-1}+o(1))\int_0^{T(1+\varepsilon)}S(\lambda)\,d\lambda.
    \]
    for $T$ sufficiently large. Similarly, for $M$, we have
    \[
    (2\alpha^{-1}+o(1))\int_0^\frac{W}{1+\varepsilon}S^{-1}(\mu)\,d\mu\leq |M\cap[-W,W]|\leq(2\alpha^{-1}+o(1))\int_0^{W(1+\varepsilon)}S^{-1}(\mu)\,d\mu
    \]
    for $W$ sufficiently large. Consequently, we know that 
    \[
    \inf_{W,T\geq K}\frac{|\Lambda\cap[-T,T]|+|M\cap[-W,W]|}{4WT}\sim\inf_{W,T\geq K}\frac{2\alpha^{-1}\int_0^TS(\lambda)\,d\lambda+2\alpha^{-1}\int_0^WS^{-1}(\mu)\,d\mu}{4WT}
    \]
    as $K\to\infty$. Finally, by Lemma~\ref{lem:fenchel-young-integral}, we know that
    \[
    \inf_{W,T\geq K}\frac{2\alpha^{-1}\int_0^TS(\lambda)\,d\lambda+2\alpha^{-1}\int_0^WS^{-1}(\mu)\,d\mu}{4WT}\to\frac{1}{2\alpha}
    \]
    as $K\to\infty$.
\end{proof}
\section{Proof of Proposition~\ref{prop:widely-admissible-dominates}}
This is deceptively difficult as one has to stop the appearance of an unbounded sequence $\{x_k\}_{k\in\mathbb{N}}$ where $\frac{S'(x_k)}{S(x_k)^{\frac{3}{2}-\delta}}$ has arbitrarily large and thin "spikes". In fact for any increasing function $\varphi:[0,\infty)\to(0,\infty)$, we can construct an entire function $S$ with non-negative Maclaurin coefficients such that there exists a sequence $\{x_k\}_{k\in\mathbb{N}}$ where $\lim_{k\to\infty}x_k=\infty$ and $\limsup_{k\to\infty}\frac{S'(x_k)}{\varphi(x_k)}=\infty$.

We begin by constructing a larger function which is entire using standard complex analytic methods. We then construct an even larger function satisfying condition (ii) through a dyadic decomposition of the Maclaurin coefficients.
\begin{proof}[Proof of Proposition~\ref{prop:widely-admissible-dominates}]
   By the Mittag--Leffler theorem, we can construct a meromorphic function $\tilde{g}(z)$ having simple poles at each $z = n \in \mathbb{N}_0$ for which $f(n+1) \neq 0$, with residue $f(n+1)$ at $z = n$. We now define $g(z):=\sin(\pi z)\tilde{g}(z)$. We now have $g(n)=f(n+1)$ for all $n\in\mathbb{N}_0$ by construction. 
   
   We let $a_n:=\frac{g^{(n)}(0)}{n!}$ be the Maclaurin coefficients of $g$. Let now $k(n):=\lfloor\frac{\log n}{\log2}\rfloor$ and $h(x):=\sum_{n=0}^\infty b_n x^n$ where 
    \[
    b_n :=
    \begin{cases}
    \displaystyle \max_{j\geq2^{k(n)}} |a_j|, & n > 0, \\[1em]
    \displaystyle \max_{j \ge 0} |a_j|, & n = 0.
    \end{cases}
    \]

    We know that 
    \[
    \sup_{j\geq n}\sqrt[j]{|b_j|}\leq\sup_{j\geq \frac{n}{2}}\sqrt[j]{|a_j|}\to 0\quad\text{as $n\to\infty$,}
    \]
     so $h$ is entire by the Cauchy--Hadamard theorem. We also know that $h(x)$ and $f(x)$ are non-decreasing so since $h(n)\geq g(n)=f(n+1)$ for all $n\in\mathbb{N}_0$, we have $h(x)\geq h(\lfloor x\rfloor)\geq f(\lceil x\rceil)\geq f(x)$ when $x\geq0$.
    
    Now let $\tilde{h}(x):=\sum_{n=0}^\infty c_n x^n$ where 
    \[
    c_n :=
    \begin{cases}
    \displaystyle \max_{j\geq0}\sqrt[2^j]{2^{(2^j-1)(k(n)+5)-j}b_{2^{k(n)+j}}}, & n > 0, \\ 
    \max\{b_0,\sqrt{c_1}\}, & n = 0,
    \end{cases}
    \]
    we notice that $\tilde{h}(x)$ is also entire by the Cauchy--Hadamard theorem. We now have $c_n^2\geq 2^{k(n)+4} c_{2n}\geq 8nc_{2n}$ for all $n>0$. Using the Cauchy product of $\tilde{h}(x)^2$, we see that
    \[
    \tilde{h}(x)^2=\sum_{n=0}^\infty x^n\sum_{k=0}^\infty c_{n-k}c_k\geq c_1+\sum_{n=1}^\infty c_n^2x^{2n}\geq c_1+\sum_{n=1}^\infty 8n c_{2n}x^{2n}.
    \]
    Since $c_n=c_j$ for all $2^{k(n)}\leq j<2^{k(n)+1}$, we know that $c_{2n}=c_{2n+1}$ for all $n\geq1$ since they are in the same dyadic block. Hence, we know that $8nc_{2n}= 4n(c_{2n}+c_{2n+1})\geq 2n c_{2n}+(2n+1)c_{2n+1}$ for all $n>0$. Consequently,
    \[
    \tilde{h}(x)^2\geq c_1+ \sum_{n=1}^\infty (2nc_{2n}+(2n+1)c_{2n+1})x^{2n}\geq \tilde{h}'(x)
    \]
    for all $x\geq1$. Hence, $\displaystyle{\limsup_{x\to\infty}}\,\tilde{h}'(x)\tilde{h}(x)^{-2}\leq 1$.
    
    Consider now $S(x):=e^{\tilde{h}(x)}$. Then 
    \[
    \limsup_{x\to\infty}\frac{S'(x)}{S(x)^{\frac{3}{2}-\delta}}= \limsup_{x\to\infty}\frac{\tilde{h}'(x)}{e^{(\frac{1}{2}-\delta)\tilde{h}(x)}}=\limsup_{x\to\infty}\frac{\tilde{h}'(x)}{\tilde{h}(x)^2}\frac{\tilde{h}(x)^2}{e^{(\frac{1}{2}-\delta)\tilde{h}(x)}}\leq \limsup_{x\to\infty}\frac{\tilde{h}(x)^2}{e^{(\frac{1}{2}-\delta)\tilde{h}(x)}}=0.
    \]
    So $S$ is widely admissible, $\tilde{h}(x)\geq h(x)\geq f(x)$ when $x\geq0$, and $\tilde{h}(x)=o(S(x))$ as $\displaystyle{x\to\infty}$.
\end{proof}
\section{Proof of Theorem~\ref{thm:uniform-critical-uniqueness}}
Our methods are inspired by a combination of techniques used by Kulikov, Nazarov, and Sodin~\cite{kulikov2025}. However, due to our uniform density assumptions, we do not require the use of complex analytic arguments.
\begin{definitionx}\label{def:ell-dense-set}
    Let $\Gamma\subset\mathbb{R}$. If $\mathbb{R}\backslash\Gamma$ contains no interval of length greater than $\ell$, we call $\Gamma$ an \emph{$\ell$-dense set}.
\end{definitionx}

We will use the following version of the well-known Poincar\'{e}--Wirtinger inequality to prove Theorem~\ref{thm:uniform-critical-uniqueness}.
\begin{lemmax}[Lemma 3.1(ii) in \cite{kulikov2025}]\label{lem:poincare-wirtinger-density}
 Let $t > 0$, and let $\Gamma$ be a $(2t)^{-1}$-dense discrete subset of $\mathbb{R}$. Then, for all convex increasing $C^1$-functions $\Phi \colon [0,\infty) \to \mathbb{R}$ and all functions $f \in \mathcal{H}$ vanishing on $\Gamma$, we have
\[
\Phi(t^2) \int_{\mathbb{R}} |f(x)|^2 \, dx 
\leq \int_{\mathbb{R}} \Phi(\xi^2) |\hat{f}(\xi)|^2 \, d\xi .
\]
\end{lemmax}
Before stating our next result, we define the functions
\begin{align*}
    \sigma_{\Phi, H}(t)&:=\sqrt{
\Phi^{-1}\left(\int_\mathbb{R}|\hat{H}(\eta)|^2\Phi((\max\{|t+2|,4\}+|\eta|)^2)\,d\eta\right)}\text{ and}\\
S_{\Phi,\Psi}(t)&:=\sqrt{\Phi^{-1}\circ\Psi(t^2)},
\end{align*}
where $\Psi$ and $\Phi$ are convex, increasing functions and $H\in\mathcal{S}$. The following lemma can be viewed as a kind of fractional, global Poincar\'{e}--Wirtinger inequality.
\begin{lemma}\label{lem:fractional-poincare-wirtinger}
Let $F\in\mathcal{S}$ and $G\in C^\infty_c(-1,1)$ with $\|G\|_{L^2}=1$. If $f$ vanishes on the set $\Lambda=\{\lambda_j\}_{j\in\mathbb{Z}}$ and
\begin{align*}
\sup_{j\in\mathbb{Z}} S_{\Phi,\Psi}(\sigma_{\Psi,F}(|\lambda_j|))(\lambda_{j+1}-\lambda_j)&\leq \frac{1}{2},
\end{align*}
 then we have the inequality
 \[
 \iint_{\mathbb{R}\times\mathbb{R}}\Psi((x+y)^2)|f(x)|^2|\hat{F}(y)|^2\,dy\,dx\leq\iint_{\mathbb{R}\times\mathbb{R}}\Phi((\xi+\eta)^2)|\hat{f}(\xi)|^2|\hat{G}(\eta)|^2\,d\xi\,d\eta.
 \] 
 Additionally, assuming the right-hand side is finite, we have equality if and only if $f\equiv 0$.
\end{lemma}
\begin{proof}
If $f$ vanishes on $\Lambda$, we have that $G_v f$ vanishes on a $(2S_{\Phi,\Psi}(\sigma_{\Psi,F}(|v|-1)))^{-1}$-dense set, where $G_v=G(x-v)$. We define the convex, increasing function
\[
\Phi_T(t):=\begin{cases}
    \Phi(t), & t\leq T,\\
    \Phi(T)+\Phi'(T)(t-T), & t>T.
\end{cases}
\]
By Lemma~\ref{lem:poincare-wirtinger-density}, we get the inequality
\[
\Phi_T(S_{\Phi,\Psi}(
\sigma_{\Psi,F}(|v|-1))^2)\int_\mathbb{R}|G_vf(x)|^2\,dx\leq\int_\mathbb{R}\Phi_T(\xi^2)|\hat{f}*\hat{G_v}|^2\,d\xi.
\]
Integrating over $v$, we get 
\begin{align}
    \int_\mathbb{R}\Phi_T(S_{\Phi,\Psi}(\sigma_{\Psi,F}(|v|-1))^{2})\int_\mathbb{R}|G_vf(x)|^2\,dx\,dv\leq\iint_{\mathbb{R}\times\mathbb{R}}\Phi_T(\xi^2)|\hat{f}*\hat{G_v}|^2\,d\xi\,dv.\notag   
\end{align}
Letting $T\to\infty$, we get
\begin{align}
    \int_\mathbb{R}\Phi(S_{\Phi,\Psi}(\sigma_{\Psi,F}(|v|-1))^{2})\int_\mathbb{R}|G_vf(x)|^2\,dx\,dv
    \leq \iint_{\mathbb{R}\times\mathbb{R}}\Phi(\xi^2)|\hat{f}*\hat{G_v}|^2\,d\xi\,dv
    \label{eq:integrated-poincare-inequality}
\end{align}
by the monotone convergence theorem. By the Plancherel theorem, we have
\begin{align}\label{eq:plancherel_identity}
    \iint_\mathbb{R\times \mathbb{R}}\Phi(\xi^2)|\hat{f}*\hat{G_v}|^2\,d\xi\,dv&=\iint_\mathbb{R\times \mathbb{R}}\Phi(\xi^2)\left|\int_\mathbb{R} \hat{f}(\xi-\eta)\hat{G}(\eta)e^{-2\pi i v\eta}\,d\eta\right|^2\,d\xi\,dv\\
    &=\int_\mathbb{R}\Phi(\xi^2)\int_\mathbb{R}\left|\int_\mathbb{R} \hat{f}(\xi-\eta)\hat{G}(\eta)e^{-2\pi i v\eta}\,d\eta\right|^2\,dv\,d\xi\notag\\
    &=\int_\mathbb{R}\Phi(\xi^2)\int_\mathbb{R} |\hat{f}(\xi-\eta)|^2|\hat{G}(\eta)|^2\,d\eta\,d\xi\notag\\
    &=\iint_{\mathbb{R}\times\mathbb{R}}\Phi((\xi+\eta)^2)|\hat{f}(\xi)|^2|\hat{G}(\eta)|^2\,d\xi\,d\eta.\notag
\end{align}
We now seek to lower bound the left-hand side of (\ref{eq:integrated-poincare-inequality}). By Fubini's theorem, we can rearrange the integral in the following way
\begin{align}
    &\int_\mathbb{R}\Phi(S_{\Phi,\Psi}(\sigma_{\Psi,F}(|v|-1))^2)\int_\mathbb{R}|G_vf(x)|^2\,dx\,dv\label{eq:integration-order-swap}\\
    &=\int_\mathbb{R}|f(x)|^2\int_\mathbb{R}\Phi(S_{\Phi,\Psi}(\sigma_{\Psi,F}(|v|-1))^{2})|G_v(x)|^2\,dv\,dx.\notag
\end{align}
 Since $\sigma_{\Psi,F}(t)$ is increasing when $t>2$, we know that $\sigma_{\Psi,F}(|x|-2)< \sigma_{\Psi,F}(|v|-1)$ if $\displaystyle{|x|\in (|v|-1, |v|+1)}$ and $|x|\geq4$. This gives us
 \begin{align*}
     \int_\mathbb{R}\Phi(S_{\Phi,\Psi}(\sigma_{\Psi,F}(|v|-1))^{2})|G_v(x)|^2\,dv 
    &>\Phi(S_{\Phi,\Psi}(\sigma_{\Psi,F}(|x|-2))^2)\int_\mathbb{R}|G_v(x)|^2\,dv\\
    &=\Phi(S_{\Phi,\Psi}(\sigma_{\Psi,F}(|x|-2))^2)\\
    &=\int_\mathbb{R}|\hat{F}(y)|^2\Psi((\max\{|x|,4\}+|y|)^2)\,dy\\
    &\geq\int_\mathbb{R}|\hat{F}(y)|^2\Psi((x+y)^2)\,dy
\end{align*}
when $|x|\geq4$. When $|x|<4$, we know that $\sigma_{\Psi,F}(|x|-2)\leq\sigma_{\Psi,F}(|v|-1)$ since $\sigma_{\Psi,F}(t)$ is non-decreasing for all $t\geq0$, but we now get a strict inequality at the last step since $|x|<\max\{|x|,4\}$. Consequently, we have 
 \begin{align}
     \int_\mathbb{R}\Phi(S_{\Phi,\Psi}(\sigma_{\Psi,F}(|v|-1))^{2})|G_v(x)|^2\,dv 
    >\int_\mathbb{R}|\hat{F}(y)|^2\Psi((x+y)^2)\,dy\label{eq:phi-lower-bound-shift}
\end{align}
for all $x\in\mathbb{R}$. Assuming $f\not\equiv 0$ and that the right-hand side of (\ref{eq:integrated-poincare-inequality}) is finite, we can combine (\ref{eq:integration-order-swap}) and (\ref{eq:phi-lower-bound-shift}) to get
\begin{align}
    \iint_{\mathbb{R}\times\mathbb{R}} |f(x)|^2|\hat{F}(y)|^2\Psi((x+y)^2)\,dy\,dx<\int_\mathbb{R}\Phi(S_{\Phi,\Psi}(\sigma_{\Psi,F}(|v|-1))^2)\int_\mathbb{R}|G_vf(x)|^2\,dx\,dv\label{eq:psi-weighted-lower-bound}
\end{align}
by the monotonicity of the Lebesgue integral. Applying the lower bound (\ref{eq:psi-weighted-lower-bound}) and the upper bound (\ref{eq:plancherel_identity}) to (\ref{eq:integrated-poincare-inequality}) yields
 \[
 \iint_{\mathbb{R}\times\mathbb{R}}\Psi((x+y)^2)|f(x)|^2|\hat{F}(y)|^2\,dy\,dx<\iint_{\mathbb{R}\times\mathbb{R}}\Phi((\xi+\eta)^2)|\hat{f}(\xi)|^2|\hat{G}(\eta)|^2\,d\xi\,d\eta,
 \] 
 assuming $f\not\equiv 0$ and that the right-hand side is finite.
\end{proof}
\begin{lemma}\label{lem:log-concave-derivative-bound}
    Let $f\in C^3$ be positive and increasing, $\mu$ a probability measure on $\mathbb{R}$ and \[     G(x):=f^{-1}\left(\int_\mathbb{R}f(x+|y|)\,d\mu(y)\right).   
    \]
    If $\log f'(t)$ is concave when $t\geq T_0>0$ and $G(x)<\infty$ for all $x\geq T_0$, then $G'(x)\leq 1$ for all $x\geq T_0$.
\end{lemma}
\begin{proof}
    By the inverse function theorem we know that
    \[
    (f^{-1})'\left(\int_\mathbb{R}f(x+|y|)\,d\mu(y)\right)=\frac{1}{f'\left(f^{-1}\left(\int_\mathbb{R}f(x+|y|)\,d\mu(y)\right)\right)}.
    \]
    Multiplying both sides by $\int_\mathbb{R}f'(x+|y|)\,d\mu(y)$, we get
    \[
    G'(x)=\frac{\int_\mathbb{R}f'(x+|y|)\,d\mu(y)}{f'(G(x))}.
    \]
    We now want to check if
    \[
    \int_\mathbb{R}f'(x+|y|)\,d\mu(y)\leq f'(G(x)).
    \]
    Defining $\omega(t):=f'(f^{-1}(t))$, we can rewrite this as
    \[
    \int_\mathbb{R}\omega(f(x+|y|))\,d\mu(y)\leq\omega\left(\int_\mathbb{R}f(x+|y|)\,d\mu(y)\right).
    \]
    This holds if $\omega$ is concave by Jensen's inequality; hence we want to show that $\omega(\tau)$ is concave when $\tau\geq\min\{f(T_0),\int_\mathbb{R}f(T_0+|y|)\,d\mu(y)\}=f(T_0)$. Since $\log f'(t)$ is concave when $t\geq T_0$, we know that
    \[
    (\log f'(t))''=\frac{f'''f'-(f'')^2}{(f')^2}(t)\leq 0,
    \]
    for $t\geq T_0$. We know that
    \[
    \omega''(t)=\left(\frac{f'''f'-(f'')^2}{(f')^3}\right)(f^{-1}(t)),
    \]
    so $\omega(\tau)$ is concave when $\tau\geq f(T_0)$.
\end{proof}
To ensure that $\sigma_{\Phi,H}(t)$ is well-defined, we will use the classical Beurling--Malliavin multiplier theorem \cite{beurling-malliavin}.
\begin{theoremx}[Beurling--Malliavin]\label{thm:beurling-malliavin}
Let $\omega$ be a non-negative Lipschitz function on $\mathbb{R}$ with $\int_{\mathbb{R}} \frac{\omega(x)}{1+x^{2}}\,dx < \infty$. Then for any $L > 0$, there exists a non-zero function $f \in L^{2}(\mathbb{R})$ such that $\operatorname{supp}f\subseteq[-L,L]$ and $|\hat{f}|\leq e^{-\omega}$.
\end{theoremx}
To prove Theorem~\ref{thm:uniform-critical-uniqueness}, we now want to use Lemma~\ref{lem:fractional-poincare-wirtinger} to create a contradictory inequality.
\begin{proof}[Proof of Theorem~\ref{thm:uniform-critical-uniqueness}]
Assume without loss of generality that $S$ is subexponentially admissible. We can also assume that $S(0)=0$, since
\begin{align*}
    [S(|\lambda_j|+\tilde{d})-S(0)](\lambda_{j+1}-\lambda_j)\leq S(|\lambda_j|+\tilde{d})(\lambda_{j+1}-\lambda_j)&\leq\frac{1}{2}\quad\text{and}\\
    S^{-1}(|\mu_j|+d+S(0))(\mu_{j+1}-\mu_j)\leq S^{-1}(|\mu_j|+d_1)(\mu_{j+1}-\mu_j)&\leq\frac{1}{2},
\end{align*}
where $d_1\geq d+S(0)$. 

Let $\Phi(t)=t$ and $\Psi(t)=S^2(\sqrt{t})$, so that $S_{\Phi,\Psi}(t)=S(t)$. Since $S$ is convex, non-negative, and increasing with $S(0)=0$, we know that $\Psi$ is convex, non-negative, and increasing. We first want to show that
\[
\sigma_{\Psi, F}(t)=\sqrt{\Psi^{-1}\left(\int_\mathbb{R}|\hat{F}(\eta)|^2\Psi((\max\{|t+2|,4\}+|\eta|)^2)\,d\eta\right)}\leq t+d
\]
for some $d\geq0$. Since $\int_\mathbb{R}\frac{|\log \max\{\Psi(t^2),1\}|}{1+t^2}\,dt<\infty$, we know that 
\[
\int_\mathbb{R}\frac{|\log \max\{\Psi(t^2),1\}+|\log (1+t^2)|}{1+t^2}\,dt<\infty.
\] 
We also need to show that
\[
\omega(t):=\log \max\{\Psi(t^2),1\}+\log(1+t^2)
\]
is Lipschitz. Differentiating $\omega$, we get
\[
|\omega'(t)|\leq \left|\frac{2t\Psi'(t^2)}{\Psi(t^2)}\right|+\left|\frac{2t}{1+t^2}\right|\leq \left|\frac{2t\Psi'(t^2)}{\Psi(t^2)}\right|+1
\]
when $t\geq \sqrt{\Psi^{-1}(1)}$, taking the right derivative at $t=\sqrt{\Psi^{-1}(1)}$. Similarly, we know that $|\omega'(t)|=\left|\frac{2t}{1+t^2}\right|\leq 1$ when $0\leq t\leq \sqrt{\Psi^{-1}(1)}$, taking the left derivative when $t=\sqrt{\Psi^{-1}(1)}$.

To upper bound $\left|\tfrac{2t\Psi'(t^2)}{\Psi(t^2)}\right|$, we use the fact that $\log \big(2t\Psi'(t^2)\big)$ is concave for $t\geq T_0$, by the definition of subexponentially admissible. Differentiating, we see that
\[
0\leq \frac{(2t\Psi'(t^2))'}{2t\Psi'(t^2)}\leq L
\]
for some $L>0$ when $t\geq T_0$. Rearranging, and integrating, we see that
\[
0\leq 2t\Psi'(t^2)\leq L\Psi(t^2)\text{, and consequently }
\left|\frac{2t\Psi'(t^2)}{\Psi(t^2)}\right|\leq L
\]
for $t\geq T_0$. If $\sqrt{\Psi^{-1}(1)}\leq T_0$, then we know that
\[
L_1:=\sup_{\sqrt{\Psi^{-1}(1)}\leq t\leq T_0}|\omega'(t)|<\infty
\]
by the extreme value theorem, since $\omega\in C^1([\sqrt{\Psi^{-1}(1)},\infty))$. Hence, $\omega$ is Lipschitz on $[0,\infty)$ with Lipschitz constant at most $L+1$ or $\max\{L+1,L_1\}$. Since $\omega$ is even, we can conclude that it is Lipschitz on all of $\mathbb{R}$.

We now know by the Beurling--Malliavin theorem, that there exists a function $H\in C^\infty_c(-1,1)$ such that 
\[|\hat{H}(\eta)|\leq \frac{\min\{(\Psi(\eta^2))^{-1},1\}}{1+\eta^2}.
\]
Choosing such an $H$ and normalizing it $F:=\frac{H}{\|H\|_{L^2}}$, we know that $\sigma_{\Psi,F}(t)<\infty$. Since $|\hat{F}(\eta)|^2\,d\eta$ is a probability measure and $\frac{d}{dt}\Psi(t^2)$ is log-concave for $t\geq T_0$, we know by Lemma~\ref{lem:log-concave-derivative-bound} that $\sigma_{\Psi,F}'(t)\leq1$ for $t\geq \max\{T_0,2\}$; hence there exists some $\tilde{d}$ such that $\sigma_{\Psi,F}(t)\leq t+\tilde{d}$.
Since $\Phi(t)=t$, we can choose any $G\in C^\infty_c(-1,1)$ with $\|G\|_{L^2}=1$ and get $\sigma_{\Phi,G}(t)\leq t+d$ using Lemma~\ref{lem:log-concave-derivative-bound}. This gives us
\[
\sigma_{\Psi,F}(t)\leq t+\tilde{d}\quad\text{and}\quad\sigma_{\Phi,G}(t)\leq t+d.
\]
Hence, by Lemma~\ref{lem:fractional-poincare-wirtinger}, any non-zero function in $\mathcal{H}$ vanishing on a pair $(\Lambda,M)$ satisfying (\ref{eq:sequence-ordering}),
    \begin{align*}
    \sup_{j\in \mathbb{Z}}S(|\lambda_j|+\tilde{d})(\lambda_{j+1}-\lambda_j)\leq \frac{1}{2},\quad\text{and}\quad
    \sup_{j\in\mathbb{Z}}S^{-1}(|\mu_j|+d)(\mu_{j+1}-\mu_j)\leq\frac{1}{2},
    \end{align*}
must also satisfy the inequalities
 \[
 \iint_{\mathbb{R}\times\mathbb{R}}\Psi((x+y)^2)|f(x)|^2|\hat{F}(y)|^2\,dy\,dx<\iint_{\mathbb{R}\times\mathbb{R}}(\xi+\eta)^2|\hat{f}(\xi)|^2|\hat{G}(\eta)|^2\,d\xi\,d\eta,
 \] 
and 
 \[
 \iint_{\mathbb{R}\times\mathbb{R}}(\xi+\eta)^2|\hat{f}(\xi)|^2|\hat{G}(\eta)|^2\,d\xi\,d\eta<\iint_{\mathbb{R}\times\mathbb{R}}\Psi((x+y)^2)|f(x)|^2|\hat{F}(y)|^2\,dy\,dx.
 \]

This is clearly a contradiction; hence, no such non-zero function can exist.
\end{proof}

\section{Proof of Theorem~\ref{thm:widely-admissible-uniqueness}}

\begin{lemma}\label{lem:f-shift-asymptotics}
    Let $f:[0,\infty)\to[0,\infty)$ be some convex, increasing $C^1$-function with 
    \[
    \limsup_{x\to\infty}\frac{f'(x)}{f^{1+c}(x)}=0
    \]
    for some $c>0$. Then for all $C>0$, we have
    $f\left(x+\frac{C}{f(x)^c}\right)\sim f(x)$ as $x\to\infty$.
\end{lemma}
\begin{proof}
    Let $R_x:=\frac{f\left(x+\frac{C}{f(x)^{c}}\right)}{f(x)}$.
    For any $\frac{1}{C}>\varepsilon>0$ there exists an $X>0$ such that $\frac{f'(x)}{f^{1+c}(x)}<\varepsilon$ when $x\geq X$. Consequently,
    \begin{align*}
        \frac{1}{f(x)^c}-\frac{1}{f\left(x+\frac{C}{f(x)^c}\right)^c}=\int_{x}^{x+\frac{C}{f(x)^c}}\frac{f'(y)}{f^{1+c}(y)}\,dy<\frac{C\varepsilon}{f(x)^c}.
    \end{align*}
    Multiplying both sides with $f\left(x+\frac{C}{f(x)^c}\right)^c$, we get $R_x^c-1<C\varepsilon R_x^c$ and consequently $ 1\leq R_x<\frac{1}{(1-C\varepsilon)^\frac{1}{c}}$ when $x\geq X$. Letting $\varepsilon\to0$, we see that $R_x\to1$ as $x\to\infty$. 
\end{proof}
\begin{lemmax}[Lemma 1 in \cite{kulikov2025}]\label{lem:interval-energy}
For any $\varepsilon > 0$, all $a < b$ and all functions $f \in \mathcal{H}$, we have
\[
\int_a^b |f(x)|^2\,dx \leq (1+\varepsilon) \left( \frac{b-a}{\pi} \right)^2 \int_a^b |f'(x)|^2\,dx 
+ (1+\varepsilon^{-1})(b-a)(|f(a)|^2 + |f(b)|^2).
\]
\end{lemmax}
\begin{lemma}\label{lem:fourier-weighted-gap}
    Let $S:[0,\infty)\to[0,\infty)$ be some increasing function. Let $M=(\mu_j)_{j\in\mathbb{Z}}$ and
    \[
    \sup_{j\in\mathbb{Z}}S(\max\{|\mu_j|,|\mu_{j+1}|\})(\mu_{j+1}-\mu_j)\leq\frac{1}{2}.
    \]
    Then if $f|_M=0$, we have
   \begin{align*}
       &\int_\mathbb{R}S^2(|\xi|)|\hat{f}(\xi)|^2\,d\xi\\
       &\leq(1+\varepsilon)\int_\mathbb{R}x^2|f(x)|^2\,dx+2(1+\varepsilon^{-1})\sum_{j\in\mathbb{Z}}(\mu_{j+1}-\mu_{j-1})S^2(\max\{|\mu_{j-1}|,|\mu_{j+1}|\})|\hat{f}(\mu_j)|^2.
   \end{align*}
   for all $\varepsilon>0$.
\end{lemma}
\begin{proof}
    By Lemma~\ref{lem:interval-energy}, we have
    \begin{align*}
        &\int_\mathbb{R}S^2(|\xi|)|\hat{f}(\xi)|^2\,d\xi\\
        &\leq\sum_{j\in\mathbb{Z}}S^2(\max\{|\mu_j|,|\mu_{j+1}|\})\int_{\mu_j}^{\mu_{j+1}}|\hat{f}(\xi)|^2\,d\xi\\
        &\leq\frac{1+\varepsilon}{(2\pi)^2}\int_\mathbb{R}|\hat{f}'(\xi)|^2\,d\xi+(1+\varepsilon^{-1})\sum_{j\in\mathbb{Z}}(\mu_{j+1}-\mu_j)S^2(\max\{|\mu_j|,|\mu_{j+1}|\})(|\hat{f}(\mu_j)|^2+|\hat{f}(\mu_{j+1})|^2)\\
        &\leq\frac{1+\varepsilon}{(2\pi)^2}\int_\mathbb{R}|\hat{f}'(\xi)|^2\,d\xi+2(1+\varepsilon^{-1})\sum_{j\in\mathbb{Z}}(\mu_{j+1}-\mu_{j-1})S^2(\max\{|\mu_{j-1}|,|\mu_{j+1}|\})|\hat{f}(\mu_j)|^2.
    \end{align*}
    By the Plancherel theorem, we know that 
    \[
    \frac{1+\varepsilon}{(2\pi)^2}\int_\mathbb{R}|\hat{f}'(\xi)|^2\,d\xi=(1+\varepsilon)\int_\mathbb{R}x^2|f(x)|^2\,dx
    \]
    thus finishing the proof.
\end{proof}
If $S$ grows faster than an exponential, then we cannot use Lemma~\ref{lem:log-concave-derivative-bound} and Theorem~\ref{thm:beurling-malliavin}. However, we can still construct similar inequalities using the power series representations of $S$.

\begin{proof}[Proof of Theorem~\ref{thm:widely-admissible-uniqueness}]
We can assume that $\tilde{S}$ is not a polynomial, since then we would get uniqueness from Theorem~\ref{thm:uniform-critical-uniqueness}. Furthermore, we can assume without loss of generality that $\tilde{S}^{(n)}(0)=0$ for all $n\leq \lceil\frac{2}{s}\rceil$, where $s$ and $\tilde{S}$ are as in the definition of widely admissible. This is because trivially $\tilde{S}(x)\sim \tilde{S}_1(x)$ as $x\to\infty$, where
\[
\tilde{S}_1(x):=\tilde{S}(x)-\sum_{n=0}^{\lceil\frac{2}{s}\rceil}\frac{\tilde{S}^{(n)}(0)}{n!}x^n.
\]
Similarly, since $\tilde{S}(x)\sim \tilde{S}_1(x)$ as $x\to\infty$ and both are convex and increasing, we have $\tilde{S}^{-1}(x)\sim \tilde{S}_1^{-1}(x)$ as $x\to\infty$, where $\tilde{S}^{-1}(y):=\inf\{x\geq0:\tilde{S}|_{[0,\infty)}(x)\geq y\}$.

Assume that $f|_\Lambda=0$ and $\hat{f}|_M=0$, and $(\Lambda, M)$ satisfies (\ref{eq:sequence-ordering}),
\begin{align*}
    \sup_{j\in\mathbb{Z}}\tilde{S}^{-1}(8e\max\{|\lambda_j|,|\lambda_{j+1}|\})^\frac{1}{s}(\lambda_{j+1}-\lambda_j)&\leq\frac{1}{2},\quad\text{and}\\
    \sup_{j\in\mathbb{Z}} \left(\tilde{S}(\max\{|\mu_j|,|\mu_{j+1}|\}^s)+\sqrt{C_{\tilde{S},s}(48e\pi^{-1}s^2)}\right)(\mu_{j+1}-\mu_j)&\leq\frac{1}{2},
\end{align*}
where $C_{\tilde{S},s}(48e\pi^{-1}s^2)$ is a constant depending only on $\tilde{S}$ and $s$. Clearly $S(x)\sim (S(x)+\sqrt{C_{\tilde{S},s}(48e\pi^{-1}s^2)})$ as $x\to\infty$. We also have that $\tilde{S}^{-1}(x)\sim \tilde{S}^{-1}(8ex)$ as $x\to\infty$, since $\tilde{S}$ grows faster than any polynomial.

Let
\begin{align}\label{eq:Gx_def}
    G(x):=\mathbbm{1}_{[-u,u]} * 
\underbrace{\frac{k}{u} \mathbbm{1}_{\left[-\tfrac{u}{2k}, \tfrac{u}{2k}\right]} 
* \dots * 
\frac{k}{u} \mathbbm{1}_{\left[-\tfrac{u}{2k}, \tfrac{u}{2k}\right]}}_{k\text{--fold convolution}}
\end{align}
and $G_v(x):=G(x-v)$. We know that $\operatorname{supp}(G_v) \subset \left[ v - \tfrac{3}{2}u,\; v + \tfrac{3}{2}u \right]
$ and hence $G_vf$ vanishes on a $(2\tilde{S}^{-1}(8e\left(|v|-\frac{3}{2}u\right))^\frac{1}{s})^{-1}$-dense set. Thus, applying Lemma~\ref{lem:poincare-wirtinger-density} with $\Phi(t)=t^\theta$ for some $\theta\geq 1$, we have
\begin{align*}
    &\tilde{S}^{-1}\left(8e\left(|v|-\frac{3}{2}u\right)\right)^{2\frac{\theta}{s}}\int_\mathbb{R}|G_v f(x)|^2dx\leq\int_\mathbb{R}|\xi|^{2\theta}|\hat{f}*\hat{G_v}|^2\,d\xi
\end{align*}
as long as $8e\left(|v|-\frac{3}{2}u\right)\geq 0$. Integrating over $v$ we get
\begin{align*}
    &
    \int_{|v|\geq\frac{3}{2}u}\tilde{S}^{-1}\left(8e\left(|v|-\frac{3}{2}u\right)\right)^{2\frac{\theta}{s}}\int_\mathbb{R}|G_v f(x)|^2dx\leq\iint_{\mathbb{R}\times\mathbb{R}}|\xi|^{2\theta}|\hat{f}*\hat{G_v}|^2\,d\xi\,dv.
\end{align*}
Using the same calculation as in (\ref{eq:plancherel_identity}), we have
\begin{align*}
    \iint_\mathbb{R\times \mathbb{R}}|\xi|^{2\theta}|\hat{f}*\hat{G_v}|^2\,d\xi\,dv=\iint_{\mathbb{R}\times\mathbb{R}}|\xi+\eta|^{2\theta}|\hat{f}(\xi)|^2|\hat{G}(\eta)|^2\,d\xi\,d\eta
\end{align*}
by the Plancherel theorem, giving us the inequality
\begin{align}
    \int_{|v|\geq\frac{3}{2}u}\tilde{S}^{-1}\left(8e\left(|v|-\frac{3}{2}u\right)\right)^{2\frac{\theta}{s}}\int_\mathbb{R}|G_v f(x)|^2dx\,dv \leq\iint_{\mathbb{R}\times\mathbb{R}}|\xi+\eta|^{2\theta}|\hat{f}(\xi)|^2|\hat{G}(\eta)|^2d\xi d\eta.\label{eq:v-integrated-fourier-bound}
\end{align}
Since $\tilde{S}^{-1}$ is non-decreasing, we know that $\tilde{S}^{-1}(8e(|x|-3u))\leq \tilde{S}^{-1}(8e\left(|v|-\frac{3}{2}u\right))$, when $|x|\in[|v|-\frac{3}{2}u,|v|+\frac{3}{2}u]$. This gives us the lower bound
\begin{align}
&\int_\mathbb{R}|G(v)|^2dv\int_{|x|\geq 3u}\tilde{S}^{-1}(8e(|x|-3u))^{2\frac{\theta}{s}}\ |f(x)|^2dx\label{eq:weighted-s-inverse-inequality}\\
&\leq\int_{|v|\geq\frac{3}{2}u}\tilde{S}^{-1}\left(8e\left(|v|-\frac{3}{2}u\right)\right)^{2\frac{\theta}{s}}\int_\mathbb{R}|G_v f(x)|^2dx\,dv.\notag
\end{align}
Since
\[
\widehat{G}(\eta) 
= \widehat{\mathbbm{1}}_{[-u,u]}(\eta)  
\left( \frac{\sin\!\left( \pi \tfrac{u}{k} \eta \right)}{\pi \tfrac{u}{k} \eta} \right)^k\quad\text{and}\quad\bigl(|\xi| + |\eta|\bigr)  
\left| \frac{\sin\!\left(\pi \tfrac{u}{k} \eta\right)}{\pi \tfrac{u}{k} \eta} \right|
\;\leq\; |\xi| + \frac{k}{\pi u},
\]
we know that
\begin{align}
    \iint_{\mathbb{R}\times\mathbb{R}}|\xi+\eta|^{2\theta}|\hat{f}(\xi)|^2|\hat{G}(\eta)|^2d\xi d\eta\leq2u\int_\mathbb{R}\left(|\xi|+\frac{k}{\pi u}\right)^{2\theta}|\hat{f}(\xi)|^2d\xi\label{eq:fourier-convolution-bound}
\end{align}
assuming $k\geq\theta$. Hence, we choose $k$ as $k=\lceil\theta\rceil$. We also know that 
\[
\underbrace{\frac{\lceil\theta\rceil}{u} \mathbbm{1}_{\left[-\tfrac{u}{2\lceil\theta\rceil}, \tfrac{u}{2\lceil\theta\rceil}\right]} 
* \dots * 
\frac{\lceil\theta\rceil}{u} \mathbbm{1}_{\left[-\tfrac{u}{2\lceil\theta\rceil}, \tfrac{u}{2\lceil\theta\rceil}\right]}}_{\lceil\theta\rceil\text{--fold convolution}}
\]
is a probability measure with support $[-\frac{u}{2},\frac{u}{2}]$. Hence, $G_v(x)=1$ when $|x-v|\leq\frac{u}{2}$, so
$u\leq\displaystyle{\int_\mathbb{R}|G(x-v)|^2\,dv}$. Applying the lower bound (\ref{eq:weighted-s-inverse-inequality}) and the upper bound (\ref{eq:fourier-convolution-bound}) to (\ref{eq:v-integrated-fourier-bound}), we get
\begin{align*}
    &u\int_{|x|\geq 3u}\tilde{S}^{-1}(8e(|x|-3u))^{2\frac{\theta}{s}}\ |f(x)|^2dx\leq2u\int_\mathbb{R}\left(|\xi|+\frac{\lceil\theta\rceil}{\pi u}\right)^{2\theta}|\hat{f}(\xi)|^2d\xi,
\end{align*}
which we can divide by $u$ to get
\begin{align*}
    &\int_{|x|\geq 3u}\tilde{S}^{-1}(8e(|x|-3u))^{2\frac{\theta}{s}}\ |f(x)|^2dx\leq2\int_\mathbb{R}\left(|\xi|+\frac{\lceil\theta\rceil}{\pi u}\right)^{2\theta}|\hat{f}(\xi)|^2d\xi.
\end{align*}
For $|x|\geq 6u$, we have
\[
\tilde{S}^{-1}(4e|x|)= \tilde{S}^{-1}\left(8e|x|\left(\frac{6u-3u}{6u}\right)\right)\leq \tilde{S}^{-1}(8e(|x|-3u))
\]
and consequently
\begin{align}
    \int_{|x|\geq 6u}\tilde{S}^{-1}(4e|x|)^{2\frac{\theta}{s}}\ |f(x)|^2dx\leq2\int_\mathbb{R}\left(|\xi|+\frac{\lceil\theta\rceil}{\pi u}\right)^{2\theta}|\hat{f}(\xi)|^2d\xi.\label{eq:weighted-moment-inequality}
\end{align}
Choosing $u=\frac{2\lceil\theta\rceil^2}{\pi}$, we know that 
\begin{align}
    \int_{|\xi|\geq1}\left(|\xi|+\frac{k}{\pi u}\right)^{2\theta}|\hat{f}(\xi)|^2\,d\xi&\leq\left(1+\frac{1}{2\lceil\theta\rceil}\right)^{2\theta}\int_{|\xi|\geq1}|\xi|^{2\theta}|\hat{f}(\xi)|^2\,d\xi\label{eq:fourier-high-freq-bound}\\
    &\leq\left(1+\frac{1}{2 \lceil\theta\rceil}\right)^{2\lceil\theta\rceil}\int_{|\xi|\geq1}|\xi|^{2\theta}|\hat{f}(\xi)|^2\,d\xi\notag\\
    &\leq e\int_{|\xi|\geq1}|\xi|^{2\theta}|\hat{f}(\xi)|^2\,d\xi\notag
\end{align}
since $(1+\frac{1}{x})^x$ is monotonically increasing for $x>0$ and converges to $e$ as $x\to\infty$. Similarly, we have
\begin{align}
        \int_{|\xi|\leq1}\left(|\xi|+\frac{\lceil\theta\rceil}{\pi u}\right)^{2\theta}|\hat{f}(\xi)|^2\,d\xi\leq e\int_{|\xi|\leq1}|\hat{f}(\xi)|^2\,d\xi.\label{eq:fourier-low-freq-bound}
\end{align}
Combining (\ref{eq:weighted-moment-inequality}) with the upper bounds (\ref{eq:fourier-high-freq-bound}) and (\ref{eq:fourier-low-freq-bound}) yields
\begin{align*}
    \int_{|x|\geq 6u}\tilde{S}^{-1}(4e|x|)^{2\frac{\theta}{s}}\ |f(x)|^2dx\leq2e\left(\int_\mathbb{R}\left|\xi\right|^{2\theta}|\hat{f}(\xi)|^2d\xi+\int_\mathbb{R}|\hat{f}(\xi)|^2\,d\xi\right).
\end{align*} 
Choosing $\theta$ such that $n=2\frac{\theta}{s}\in\mathbb{N}$ and multiplying both sides by $\frac{A_n}{n!}\geq 0$ where $A_n=\frac{d^n}{dz^n}\tilde{S}^2(0)$, we have
\begin{align}
    \int_{|x|\geq 6u}\frac{A_n\tilde{S}^{-1}(4e|x|)^{n}}{n!}\ |f(x)|^2dx\leq\frac{2eA_n}{n!}\left(\int_\mathbb{R}\left|\xi\right|^{ns}\hat{f}(\xi)|^2d\xi+\int_\mathbb{R}|\hat{f}(\xi)|^2\,d\xi\right).\label{eq:nth-moment-bound}
\end{align}
Since 
\[
\frac{A_n}{n!}\int_{|x|\leq 6u}\tilde{S}^{-1}(4e|x|)^n|f(x)|^2\,dx\leq \frac{A_n( \tilde{S}^{-1}(24eu))^n }{n!}\int_\mathbb{R}|f(x)|^2\,dx
\]
we can add this to (\ref{eq:nth-moment-bound}) and get
\begin{align*}
    &\int_\mathbb{R}\frac{A_n \tilde{S}^{-1}(4e|x|)^{n}}{n!}\ |f(x)|^2dx\\
    &\leq\frac{2eA_n}{n!}\left(\int_\mathbb{R}\left|\xi\right|^{ns}|\hat{f}(\xi)|^2d\xi+\int_\mathbb{R}|\hat{f}(\xi)|^2\,d\xi\right) +\frac{A_n(\tilde{S}^{-1}(24eu))^n}{n!}\int_\mathbb{R}|\hat{f}(\xi)|^2\,d\xi.
\end{align*}
Since $\theta\geq 1$, we know that $\lceil\theta\rceil\leq 2\theta$, hence $u=\frac{2\lceil\theta\rceil^2}{\pi}\leq\frac{8\theta^2}{\pi}=\frac{2}{\pi}(ns)^2$. By a standard Cauchy estimate, we know that
\begin{align}
    \frac{A_n}{n!}(\tilde{S}^{-1}(24eu))^n&\leq\frac{A_n}{n!}\left(\tilde{S}^{-1}\left(\frac{48e}{\pi}(ns)^2\right)\right)^n\label{eq:cauchy-estimate-bound}\\
    &\leq \max_{|z|=r}\frac{|\tilde{S}^2(z)|}{r^n}\left(\tilde{S}^{-1}\left(\frac{48e}{\pi}(ns)^2\right)\right)^n\notag
\end{align}
for all $r>0$. We want to choose $r=\left(1+\frac{4\log n}{n}\right)\tilde{S}^{-1}\left(\frac{48e}{\pi}(ns)^2\right)$. By Lemma~\ref{lem:f-shift-asymptotics}, we know that
\[
\tilde{S}\left(x+\frac{C}{\tilde{S}(x)^{\frac{1}{2}-\delta}}\right)\sim \tilde{S}(x)
\]
as $x\to\infty$ for some $0<\delta<\frac{1}{2}$. Since $S$ is convex and grows faster than any polynomial, we know that for sufficiently large $x$
\[
\frac{x\log \tilde{S}(x)}{\sqrt{\tilde{S}(x)}}\leq \frac{C}{\tilde{S}(x)^{\frac{1}{2}-\delta}}
\]
for all $0<\delta<\frac{1}{2}$ and $0<C$. This means that
\[
\tilde{S}(x)\leq \tilde{S}\left(x+\frac{4x\log x}{\sqrt{\tilde{S}(x)}}\right)\leq \tilde{S}\left(x+\frac{1}{\tilde{S}(x)^{\frac{1}{2}-\delta}}\right)
\]
for sufficiently large $x$. Consequently, choosing $r=\left(1+\frac{4\log n}{n}\right)\tilde{S}^{-1}\left(\frac{48e}{\pi}(ns)^2\right)$, we have $\frac{A_n}{n!}(S^{-1}(24eu))^n=O(n^{-2})$ as $n\to\infty$. Meaning that, we have
\begin{align*}
    &\sum_{n=\lceil\frac {2}{s}\rceil}^\infty\int_\mathbb{R}\frac{A_n\tilde{S}^{-1}(4e|x|)^{n}}{n!}\ |f(x)|^2dx
    \\&\leq2e\left(\sum_{n=\lceil\frac {2}{s}\rceil}^\infty\frac{A_n}{n!}\int_\mathbb{R}\left|\xi\right|^{ns}|\hat{f}(\xi)|^2d\xi+C_{\tilde{S},s}(48e\pi^{-1}s^2)\int_\mathbb{R}|\hat{f}(\xi)|^2\,d\xi \right)
\end{align*}
where
\begin{align}
    C_{\tilde{S},s}(t):=\tilde{S}^2(1)+\frac{1}{2e}\sum_{n=\lceil\frac {2}{s}\rceil}^\infty\frac{A_n}{n!}\left(\tilde{S}^{-1}\left(tn^2\right)\right)^n.\label{eq:CStilde}
\end{align}
By a similar calculation, we know that $C_{\tilde{S},s}(t)<\infty$ for all $t>0$. Using the power series representation of $\tilde{S}^2(z)$ and the monotone convergence theorem, we get
\begin{align*}
    &4e\int_\mathbb{R}|x|^2 |f(x)|^2dx\leq2e\int_\mathbb{R}\left(\tilde{S}^2\left(\left|\xi\right|^s\right)+C_{\tilde{S},s}(48e\pi^{-1}s^2)\right)|\hat{f}(\xi)|^2d\xi.
\end{align*} 
Dividing by $2e$, we get
\begin{align}
    &2\int_\mathbb{R}x^2 |f(x)|^2dx\leq\int_\mathbb{R}\left(S^2\left(\left|\xi\right|\right)+C_{\tilde{S},s}(48e\pi^{-1}s^2)\right)|\hat{f}(\xi)|^2d\xi.\label{eq:weighted-uncertainty-principle}
\end{align} 
By Lemma~\ref{lem:fourier-weighted-gap}, we know that
\[
\int_\mathbb{R}(S^2(|\xi|)+C_{\tilde{S},s}(48e\pi^{-1}s^2))|\hat{f}(\xi)|^2\,d\xi\leq\int_\mathbb{R}x^2|f(x)|^2\,dx.
\]
Combining this with (\ref{eq:weighted-uncertainty-principle}), we get
\[
2\int_\mathbb{R}x^2|f(x)|^2\,dx\leq\int_\mathbb{R}x^2|f(x)|^2\,dx,
\]
which is a contradiction unless $f\equiv0$.
\end{proof}
\section{Proof of Theorem~\ref{thm:main-sampling-inequality}}
\subsection{Lower frame bound}
We now provide a related inequality to Lemma 3(i) in \cite{kulikov2025}, however with different tradeoffs. We get a much better weight for $\Phi$ with very rapidly growing derivatives:
\[
\frac{\Phi(t^2)-\Phi(0)}{t}\leq t\Phi'(t^2),
\]
but with an extra $\varepsilon$-dependent coefficient on the right-hand side of the inequality.
\begin{lemma}\label{lem:pw-density-sampled}
    Let $t > 0$, $0 < \varepsilon < 1$, and let $\Gamma$ be a $(1-\varepsilon)(2t)^{-1}$-dense discrete subset of $\mathbb{R}$. There exists a positive constant $C_\varepsilon$ such that, for all convex increasing functions $\Phi: [0,\infty) \to [0,\infty)$ and all functions $f \in \mathcal{H}$, we have
\[
\Phi(t^2) \int_{\mathbb{R}} |f(x)|^2 \, dx 
\leq C_\varepsilon\left(\int_{\mathbb{R}} \Phi(\xi^2) |\hat{f}(\xi)|^2 \, d\xi 
+ (\Phi(t^2)-\Phi(0)) \sum_{j\in\mathbb{Z}}(\gamma_{j+1}-\gamma_{j-1}) |f(\gamma_j)|^2\right).
\]
\end{lemma}
\begin{proof}
    Applying Lemma~\ref{lem:interval-energy} and the Plancherel theorem, we get
    \begin{align*}
        \int_\mathbb{R}|f(x)|^2\,dx&=
        \sum_{j\in\mathbb{Z}}\int_{\gamma_j}^{\gamma_{j+1}}|f(x)|^2\,dx\\
        &\leq\sum_{j\in\mathbb{Z}}\left((1+\varepsilon)\Bigg(\frac{\gamma_{j+1}-\gamma_j}{\pi}\right)^2\int_{\gamma_j}^{\gamma_{j+1}}|f'(x)|^2\,dx\\
        &+2(1+\varepsilon^{-1})(\gamma_{j+1}-\gamma_j)(|f(\gamma_j)|^2+|f(\gamma_{j+1})|^2)\Bigg)\\
        &\leq(1+\varepsilon)\left(\frac{1-\varepsilon}{2\pi t}\right)^2\int_\mathbb{R}|f'(x)|^2\,dx+2(1+\varepsilon^{-1})\sum_{j\in\mathbb{Z}}(\gamma_{j+1}-\gamma_{j-1})|f(\gamma_j)|^2\\
        &\leq\frac{1-\varepsilon}{t^2}\int_\mathbb{R}\xi^2|\hat{f}(\xi)|^2\,d\xi+2(1+\varepsilon^{-1})\sum_{j\in\mathbb{Z}}(\gamma_{j+1}-\gamma_{j-1})|f(\gamma_j)|^2.
    \end{align*}
    Multiplying by $\Phi(t^2)-\Phi(0)$, we get
    \begin{align*}
        &(\Phi(t^2)-\Phi(0))\int_\mathbb{R}|f(x)|^2\,dx\\
        &\leq(1-\varepsilon)\frac{\Phi(t^2)-\Phi(0)}{t^2}\int_\mathbb{R}\xi^2|\hat{f}(\xi)|^2\,d\xi+2(1+\varepsilon^{-1})(\Phi(t^2)-\Phi(0)) \sum_{j\in\mathbb{Z}}(\gamma_{j+1}-\gamma_{j-1}) |f(\gamma_j)|^2.
    \end{align*}
    Since $\Phi(t)-\Phi(0)$ is convex, increasing and vanishes at the origin, we know that 
    \[
    (\Phi(t^2)-\Phi(0))\frac{\xi^2}{t^2}\leq\max\{\Phi(t^2),\Phi(\xi^2)\}-\Phi(0)\leq\Phi(t^2)+\Phi(\xi^2)-2\Phi(0)
    \]
    for all $t>0$. Plugging this into our previous inequality, we get
\begin{align*}
        (\Phi(t^2)-\Phi(0))\int_\mathbb{R}|f(x)|^2\,dx
        &\leq(1-\varepsilon)\int_\mathbb{R}(\Phi(t^2)+\Phi(\xi^2)-2\Phi(0))|\hat{f}(\xi)|^2\,d\xi\\
        &+2(1+\varepsilon^{-1})(\Phi(t^2)-\Phi(0)) \sum_{j\in\mathbb{Z}}(\gamma_{j+1}-\gamma_{j-1}) |f(\gamma_j)|^2.
    \end{align*}
    Rearranging, we get 
    \begin{align*}
        &\varepsilon(\Phi(t^2)-\Phi(0))\int_\mathbb{R}|f(x)|^2\,dx\\
        &\leq\int_\mathbb{R}((\Phi(\xi^2)-\Phi(0))|\hat{f}(\xi)|^2\,d\xi+2(1+\varepsilon^{-1})(\Phi(t^2)-\Phi(0)) \sum_{j\in\mathbb{Z}}(\gamma_{j+1}-\gamma_{j-1}) |f(\gamma_j)|^2.
    \end{align*}
    Since $\varepsilon\Phi(0)\leq\Phi(0)$, we get the inequality
    \begin{align*}
        &\Phi(t^2)\int_\mathbb{R}|f(x)|^2\,dx\\
        &\leq\varepsilon^{-1}\int_\mathbb{R}\Phi(\xi^2)|\hat{f}(\xi)|^2\,d\xi
        +2\varepsilon^{-1}(1+\varepsilon^{-1})(\Phi(t^2)-\Phi(0)) \sum_{j\in\mathbb{Z}}(\gamma_{j+1}-\gamma_{j-1}) |f(\gamma_j)|^2.\qedhere
    \end{align*}
\end{proof}
We prove the lower frame bound by adapting the argument of Theorem~\ref{thm:widely-admissible-uniqueness}. 
\begin{lemma}\label{lem:lower-frame}
        Let $S$ be widely admissible where $\tilde{S}$ is not a polynomial and $0<\alpha<\frac{1}{2}$. Then there exist constants $K_{\alpha,S},C>0$ such that if $(\Lambda,M)$ satisfies (\ref{eq:sequence-ordering}) and
    \begin{align*}
        \sup_{j\in\mathbb{Z}}\max\{S^{-1}(|\lambda_j|),K_{\alpha,S}\}(\lambda_{j+1}-\lambda_j),\,\sup_{j\in\mathbb{Z}}\max\{S(|\mu_j|),K_{\alpha,S}\}(\mu_{j+1}-\mu_j)\leq \alpha
    \end{align*}
    then for any $f\in\mathcal{H}$, we have
    \begin{align*}
    C^{-1}\|f\|_{H(S)}^2
    \leq \sum_{j\in\mathbb{Z}}(\lambda_{j+1}-\lambda_{j-1})(1+|\lambda_j|^2)|f(\lambda_j)|^2 + \sum_{j\in\mathbb{Z}}(\mu_{j+1}-\mu_{j-1})(1+S^2(|\mu_j|))|\hat{f}(\mu_j)|^2.
    \end{align*}
\end{lemma}
\begin{proof}
Using the same reasoning as in the proof of Theorem~\ref{thm:widely-admissible-uniqueness}, we assume that $\tilde{S}^{(n)}(0)=0$ for all $n\leq\lceil\frac{2}{s}\rceil$. Furthermore, we assume that $(\Lambda, M)$ satisfies (\ref{eq:sequence-ordering}),
\begin{align*}
    \sup_{j\in\mathbb{Z}}\tilde{S}^{-1}(8eC_\varepsilon|\lambda_j|)^\frac{1}{s}(\lambda_{j+1}-\lambda_j)&\leq\frac{1-\varepsilon}{2},\quad\text{and}\\
    \sup_{j\in\mathbb{Z}} \left(\tilde{S}(\max\{|\mu_j|,|\mu_{j+1}|\}^s)+\sqrt{C_{\tilde{S},s}(48eC_\varepsilon\pi^{-1}s^2)}\right)(\mu_{j+1}-\mu_j)&\leq\frac{1-\varepsilon}{2},
\end{align*}
where $\tilde{S}^{-1}(y):=\inf\{x\geq0:\tilde{S}|_{[0,\infty)}(y)\geq x\}$ for some $\varepsilon>0$ satisfying $0<\varepsilon<1-2\alpha$ and $C_{\tilde{S},s}$ is defined as in (\ref{eq:CStilde}).
We can do this because for any $\delta>0$ satisfying $1+\delta\leq\frac{1-\varepsilon}{2\alpha}$ and widely admissible $S$ where $\tilde{S}$ is not a polynomial there exists a $K>0$ depending only on $S$, $\varepsilon$ and $\delta$ such that
\begin{align*}
    \sup_{j\in\mathbb{Z}}\tilde{S}^{-1}(8eC_\varepsilon|\lambda_j|)^\frac{1}{s}\leq\sup_{j\in\mathbb{Z}}(1+\delta)\max\{\tilde{S}^{-1}(|\lambda_j|)^\frac{1}{s},K\}&\leq\frac{1-\varepsilon}{2}\quad\text{and}\\
    \sup_{j\in\mathbb{Z}} (\tilde{S}(\max\{|\mu_j|,|\mu_{j+1}|\}^s))\leq\sup_{j\in\mathbb{Z}}(1+\delta)\max\{\tilde{S}(|\mu_j|^s),K\}&\leq\frac{1-\varepsilon}{2}.
\end{align*}
Applying Lemma~\ref{lem:pw-density-sampled} with $\Phi(t)=t^\theta$ for some $\theta\geq1$, we get
\begin{align*}
    &\tilde{S}^{-1}\left(8eC_\varepsilon\left(|v|-\frac{3}{2}u\right)\right)^{2\frac{\theta}{s}}\int_\mathbb{R}|G_v f(x)|^2dx\\
    &\leq C_\varepsilon\left(\int_\mathbb{R}|\xi|^{2\theta}|\hat{f}*\hat{G_v}|^2\,d\xi+\tilde{S}^{-1}\left(8e C_\varepsilon\left(|v|-\frac{3}{2}u\right)\right)^{\frac{2\theta}{s}}\sum_{j\in\mathbb{Z}}(\lambda_{j+1}-\lambda_{j-1})|G_vf(\lambda_j)|^2\right),
\end{align*}
where $G$ is defined as in (\ref{eq:Gx_def}) and $k$ and $u$ are parameters of $G$. Integrating over $|v|\geq\frac{3}{2}u$, we get
\begin{align*}
    &\int_{|v|\geq\frac{3}{2}u}\tilde{S}^{-1}\left(8eC_\varepsilon\left(|v|-\frac{3}{2}u\right)\right)^{2\frac{\theta}{s}}\int_\mathbb{R}|G_v f(x)|^2\,dx\,dv\\
    &\leq C_\varepsilon\Bigg(\iint_{\mathbb{R}\times\mathbb{R}}|\xi|^{2\theta}|\hat{f}*\hat{G_v}|^2\,d\xi\,dv\\
    &+\int_{|v|\geq\frac{3}{2}u}\tilde{S}^{-1}\left(8eC_\varepsilon\left(|v|-\frac{3}{2}u\right)\right)^{\frac{2\theta}{s}}\sum_{j\in\mathbb{Z}}(\lambda_{j+1}-\lambda_{j-1})|G_vf(\lambda_j)|^2\,dv\Bigg).
\end{align*}
We repeat the steps from the proof of Theorem~\ref{thm:widely-admissible-uniqueness} and get
    \begin{align*}
    &u\int_{|x|\geq 6u}\tilde{S}^{-1}(4eC_\varepsilon|x|)^{2\frac{\theta}{s}}\ |f(x)|^2dx\\
    &\leq2eu C_\varepsilon\left(\int_\mathbb{R}\left|\xi\right|^{2\theta}|\hat{f}(\xi)|^2d\xi
    +\int_\mathbb{R}|\hat{f}(\xi)|^2\,d\xi\right)\\
    &+C_\varepsilon \int_{|v|\geq\frac{3}{2}u}\tilde{S}^{-1}\left(8e C_\varepsilon\left(|v|-\frac{3}{2}u\right)\right)^{\frac{2\theta}{s}}\sum_{j\in\mathbb{Z}}(\lambda_{j+1}-\lambda_{j-1})|G_vf(\lambda_j)|^2\,dv.
\end{align*}
Since $G_v(x)\leq 1$ and $\operatorname{supp}G_v\subseteq [v-\frac{3}{2}u,v+\frac{3}{2}u]$, we know that $G_v(x)\leq \mathbbm{1}_{[v-\frac{3}{2}u,v+\frac{3}{2}u]}$. Hence
\begin{align*}
    &\sum_{j\in\mathbb{Z}}(\lambda_{j+1}-\lambda_{j-1})|f(\lambda_j)|^2\int_{|v|\geq\frac{3}{2}u}\tilde{S}^{-1}\left(8eC_\varepsilon\left(|v|-\frac{3}{2}u\right)\right)^\frac{2\theta}{s}|G_v(\lambda_j)|^2\,dv\\
    &\leq \sum_{j\in\mathbb{Z}}(\lambda_{j+1}-\lambda_{j-1})|f(\lambda_j)|^2\int_{|v|\geq\frac{3}{2}u}\tilde{S}^{-1}\left(8eC_\varepsilon\left(|v|-\frac{3}{2}u\right)\right)^\frac{2\theta}{s}\mathbbm{1}_{[v-\frac{3}{2}u,|v|+\frac{3}{2}u]}(\lambda_j)\,dv\\
    &\leq \sum_{j\in\mathbb{Z}}(\lambda_{j+1}-\lambda_{j-1})|f(\lambda_j)|^2\int_{\lambda_j-\frac{3}{2}u}^{\lambda_j+\frac{3}{2}u}\tilde{S}^{-1}\left(8eC_\varepsilon\left(|v|-\frac{3}{2}u\right)\right)^\frac{2\theta}{s}\,dv\\
    &\leq3u \sum_{j\in\mathbb{Z}}(\lambda_{j+1}-\lambda_{j-1})|f(\lambda_j)|^2(\tilde{S}^{-1}(8eC_\varepsilon|\lambda_j|))^\frac{2\theta}{s}.
\end{align*}
Combining these inequalities, we get
\begin{align*}
    &\int_{|x|\geq 6u}\tilde{S}^{-1}(4eC_\varepsilon|x|)^{2\frac{\theta}{s}}\ |f(x)|^2dx\\
    &\leq2e C_\varepsilon\left(\int_\mathbb{R}\left|\xi\right|^{2\theta}|\hat{f}(\xi)|^2d\xi
    +\int_\mathbb{R}|\hat{f}(\xi)|^2\,d\xi\right)+3C_\varepsilon \sum_{j\in\mathbb{Z}}(\lambda_{j+1}-\lambda_{j-1})|f(\lambda_j)|^2\tilde{S}^{-1}(8eC_\varepsilon|\lambda_j|)^\frac{2\theta}{s}.
\end{align*} 
Using the same power series trick as in the proof of Theorem~\ref{thm:widely-admissible-uniqueness}, we get
\begin{align*}
    2\int_\mathbb{R}x^2|f(x)|^2\,dx\leq \int_\mathbb{R}(S^2(|\xi|)+C_{\tilde{S},s}(48eC_\varepsilon\pi^{-1}s^2))|\hat{f}(\xi)|^2\,d\xi+6C_\varepsilon\sum_{j\in\mathbb{Z}}(\lambda_{j+1}-\lambda_{j-1})|f(\lambda_j)|^2\lambda_j^2.
\end{align*}
By Lemma~\ref{lem:fourier-weighted-gap}, we know that
\begin{align*}
    &\int_\mathbb{R}(\tilde{S}(|\xi|^s)+\sqrt{C_{\tilde{S},s}(48eC_\varepsilon\pi^{-1}s^2)})^2|\hat{f}(\xi)|^2\,d\xi\\
    &\leq\int_\mathbb{R}x^2|f(x)|^2\,dx+(1+\varepsilon^{-1})\sum_{j\in\mathbb{{Z}}}(\mu_{j+1}-\mu_j)(\tilde{S}(\max\{|\mu_j|,|\mu_{j+1}|\}^s)\\
    &+\sqrt{C_{\tilde{S},s}(48eC_\varepsilon\pi^{-1}s^2)})^2(|\hat{f}(\mu_j)|^2+|\hat{f}(\mu_{j+1})|^2).
\end{align*}
By Lemma~\ref{lem:f-shift-asymptotics}, we know that $S(|\mu_{j-1}|)\sim S(|\mu_j|)\sim S(|\mu_{j+1}|)$ as $j\to\pm\infty$, so
\begin{align*}
    &\sum_{j\in\mathbb{{Z}}}(\mu_{j+1}-\mu_{j-1})(\tilde{S}(\max\{|\mu_{j-1}|,|\mu_{j+1}|\}^s)
    +\sqrt{C_{\tilde{S},s}(48eC_\varepsilon\pi^{-1}s^2)})^2(|\hat{f}(\mu_j)|^2+|\hat{f}(\mu_{j+1})|^2)\\
    &\leq \tilde{C} \sum_{j\in\mathbb{Z}}(\mu_{j+1}-\mu_{j-1})(1+(\tilde{S}^2(|\mu_j|^s))|\hat{f}(\mu_j)|^2
\end{align*}
for some $\tilde{C}>0$. Combining these gives the bound
\begin{align*}
    &\int_\mathbb{R}(1+x^2)|f(x)|^2\,dx+\int_\mathbb{R}(1+S^2(|\xi|))|\hat{f}(\xi)|^2\,d\xi\\
    &\leq C\left(\sum_{j\in\mathbb{Z}}(\lambda_{j+1}-\lambda_{j-1})(1+|\lambda_j|^2)|f(\lambda_j)|^2+\sum_{j\in\mathbb{{Z}}}(\mu_{j+1}-\mu_{j-1})(1+S^2(|\mu_j|))|\hat{f}(\mu_j)|^2\right)
\end{align*}
for some $C>0$. 
\end{proof}
\subsection{Upper frame bound}
Our next goal is to establish an upper frame bound.

\begin{lemmax}[Lemma 2 in \cite{kulikov2025}]\label{lem:endpoint-trace}
For all $a < b$ and all $f \in \mathcal{H}$, we have
\[
\frac{1}{b-a} |f(a)|^2 \leq \frac{2}{(b-a)^2} \int_a^b |f(x)|^2\,dx 
+ \frac{2}{3} \int_a^b |f'(x)|^2\,dx .
\]
\end{lemmax}

\begin{lemma}\label{lem:upper-frame}
    Let $S$ be a widely admissible function where $\tilde{S}$ is not a polynomial.  Let $(\Lambda,M)$ be a pair satisfying 
    \begin{align*}
    c\leq\inf_{j\in \mathbb{Z}}\max\{S^{-1}(|\lambda_j|),1\}(\lambda_{j+1}-\lambda_j),\inf_{j\in\mathbb{Z}}\max\{S(|\mu_j|),1\}(\mu_{j+1}-\mu_j)
\end{align*}
for some $0<c<1$. Then there exists some $C>0$ such that
    \begin{align*}
        &\sum_{\lambda\in\Lambda}\frac{(1+\lambda^2)|f(\lambda)|^2}{\max\{S^{-1}(|\lambda|),1\}}+\sum_{\mu\in M}S(|\mu|)|\hat{f}(\mu)|^2\\
        &\leq C\left(\int_\mathbb{R}(1+x^2)|f(x)|^2\,dx+\int_\mathbb{R}(1+S^2(|\xi|))|\hat{f}(\xi)|^2\,d\xi\right).
    \end{align*}
\end{lemma}
\begin{proof}
    Using the same reasoning as in the proof of Theorem~\ref{thm:widely-admissible-uniqueness}, we assume that $\tilde{S}^{(n)}(0)=0$ for all $n\leq\lceil\frac{2}{s}\rceil$.
    Let $\tilde{S}^{-1}(y):=\inf\{x\geq0:\tilde{S}|_{[0,\infty)}(y)\geq x\}$. Since 
    \[
    \frac{c}{\max\{\tilde{S}^{-1}(\min\{|\lambda_j|,|\lambda_{j+1}|\})^\frac{1}{s},1\}}\leq \lambda_{j+1}-\lambda_j
    \]
    for all $j\in\mathbb{Z}$, we have by Lemma~\ref{lem:endpoint-trace}
    \begin{align*}
        &\left(\frac{c}{\max\{\tilde{S}^{-1}\left(|v|+\frac{3}{2}u\right)^\frac{1}{s},1\}}\right) \sum_{\lambda\in\Lambda}|G_vf(\lambda)|^2\\&\leq 2\left(\int_\mathbb{R}|G_v f(x)|^2\,dx+\left(\frac{c}{\max\{\tilde{S}^{-1}\left(|v|+\frac{3}{2}u\right)^\frac{1}{s},1\}}\right)^2\int_\mathbb{R}|\xi|^2|\hat{f}*\hat{G}_v(\xi)|^2\,d\xi\right)\\
        &\leq4\left(\int_\mathbb{R}|G_v f(x)|^2\,dx+\left(\frac{c}{\max\{\tilde{S}^{-1}\left(|v|+\frac{3}{2}u\right)^\frac{1}{s},1\}}\right)^{2\theta}\int_\mathbb{R}|\xi|^{2\theta}|\hat{f}*\hat{G}_v(\xi)|^2\,d\xi\right)
    \end{align*}
    for any $\theta\geq1$, where $G$ is defined as in (\ref{eq:Gx_def}) and $k$ and $u$ are parameters of $G$. Multiplying both sides by 
    \[c^{-1}\left(\max\left\{\tilde{S}^{-1}\left(|v|+\frac{3}{2}u\right)^\frac{1}{s},1\right\}\right)^{2\theta},\] 
    we have
     \begin{align*}
        &\left(\max\left\{\tilde{S}^{-1}\left(|v|+\frac{3}{2}u\right)^\frac{1}{s},1\right\}\right)^{2\theta-1} \sum_{\lambda\in\Lambda}|G_v f(\lambda)|^2\\        &\leq\frac{4}{c}\left(\left(\max\left\{\tilde{S}^{-1}\left(|v|+\frac{3}{2}u\right)^\frac{1}{s},1\right\}\right)^{2\theta} \int_\mathbb{R}|G_v f(x)|^2\,dx+\int_\mathbb{R}|c\xi|^{2\theta}|\hat{f}*\hat{G}_v(\xi)|^2\,d\xi\right).
    \end{align*}
    We can now integrate over $v$ and get
    \begin{align}
        &\int_\mathbb{R}\left(\max\left\{\tilde{S}^{-1}\left(|v|+\frac{3}{2}u\right)^\frac{1}{s},1\right\}\right)^{2\theta-1} \sum_{\lambda\in\Lambda}|G_v f(\lambda)|^2\,dv\label{eq:gv-mixed-space-upper}\\        
        &\leq\frac{4}{c}\int_\mathbb{R}\Bigg(\left(\max\left\{\tilde{S}^{-1}\left(|v|+\frac{3}{2}u\right)^\frac{1}{s},1\right\}\right)^{2\theta} \int_\mathbb{R}|G_v f(x)|^2\,dx\notag\\
        &+\int_\mathbb{R}|c\xi|^{2\theta}|\hat{f}*\hat{G}_v(\xi)|^2\,d\xi\Bigg)\,dv.\notag
    \end{align}
    We can lower bound the left-hand side of the inequality in the following way
    \begin{align}
        &\int_\mathbb{R} \max\left\{\tilde{S}^{-1}\left(|v|+\frac{3}{2}u\right)^\frac{1}{s},1\right\}^{2\theta-1} \sum_{\lambda\in\Lambda}|G_v f(\lambda)|^2dv\label{eq:gv-window-weight-lower}\\
        &\geq \sum_{\lambda\in\Lambda}|f(\lambda)|^2\int_\mathbb{R} \max\left\{\tilde{S}^{-1}\left(|v|+\frac{3}{2}u\right)^\frac{1}{s},1\right\}^{2\theta-1}\mathbbm{1}_{[v-\frac{1}{2}u,v+\frac{1}{2}u]}(\lambda)\,dv\notag\\
        &= \sum_{\lambda\in\Lambda}|f(\lambda)|^2\int_{\lambda-\frac{1}{2}u}^{\lambda+\frac{1}{2}u}\max\left\{\tilde{S}^{-1}\left(|v|+\frac{3}{2}u\right)^\frac{1}{s},1\right\}^{2\theta-1}\,dv\notag\\
        &\geq u\sum_{\lambda\in\Lambda}|f(\lambda)|^2\max\{\tilde{S}^{-1}(|\lambda|+u),1\}^\frac{2\theta-1}{s}.\notag\\
        &\geq u\sum_{\lambda\in\Lambda}|f(\lambda)|^2\max\{\tilde{S}^{-1}(|\lambda|),1\}^\frac{2\theta-1}{s}.\notag
    \end{align}
    We now seek to upper bound the right-hand side of (\ref{eq:gv-mixed-space-upper}). By Fubini's theorem, we can rearrange the first integral in the following way 
    \begin{align*}
        &\int_\mathbb{R}\max\left\{\tilde{S}^{-1}\left(|v|+\frac{3}{2}u\right),1\right\}^{\frac{2\theta-1}{s}}\int_\mathbb{R}|G_v f(x)|^2\,dx\,dv\\
        &=\int_\mathbb{R}|f(x)|^2\int_\mathbb{R}\max\left\{\tilde{S}^{-1}\left(|v|+\frac{3}{2}u\right),1\right\}^{\frac{2\theta-1}{s}}|G_v(x)|^2\,dv\,dx.
    \end{align*}
    We again use that $G(x)\leq\mathbbm{1}_{[-\frac{3}{2}u,\frac{3}{2}u]}(x)$ to get the following estimate
    \begin{align*}
        &\int_\mathbb{R}\max\left\{S^{-1}\left(|v|+\frac{3}{2}u\right),1\right\}^\frac{2\theta-1}{s}|G_v(x)|^2\,dv\\
        &\leq\int_\mathbb{R}\max\left\{S^{-1}\left(|v|+\frac{3}{2}u\right),1\right\}^\frac{2\theta-1}{s}\mathbbm{1}_{[v-\frac{3}{2}u,v+\frac{3}{2}u]}(x)\,dv\\
        &\leq\int_{x-\frac{3}{2}u}^{x+\frac{3}{2}u}\max\left\{S^{-1}\left(|v|+\frac{3}{2}u\right),1\right\}^\frac{2\theta-1}{s}\,dv\\
        &\leq 3u \max\{S^{-1}(|x|+3u),1\}^\frac{2\theta-1}{s}\\
        &\leq3u\max\{\tilde{S}^{-1}(2|x|),1\}^\frac{2\theta-1}{s}+3u\tilde{S}^{-1}(6u)^\frac{2\theta-1}{s}.
    \end{align*}
    This gives us the upper bound
    \begin{align}
        &\int_\mathbb{R}\max\left\{\tilde{S}^{-1}\left(|v|+\frac{3}{2}u\right),1\right\}^{\frac{2\theta-1}{s}}\int_\mathbb{R}|G_vf(x)|^2\,dx\,dv\label{eq:gv-weighted-L2-averaging}\\
        &\leq 3u\int_\mathbb{R}\max\{\tilde{S}^{-1}(2|x|),1\}^\frac{2\theta-1}{s}|f(x)|^2\,dx+3u\int_\mathbb{R}\tilde{S}^{-1}(6u)^\frac{2\theta-1}{s}|f(x)|^2\,dx.
        \notag
    \end{align}
    To upper bound the second integral on the right-hand side of (\ref{eq:gv-mixed-space-upper}), we again use the Plancherel trick (\ref{eq:plancherel_identity}) and (\ref{eq:fourier-convolution-bound}), this gives us
    \begin{align*}
        \iint_{\mathbb{R}\times\mathbb{R}}|\xi|^{2\theta}|\hat{f}*\hat{G}_v(\xi)|^2\,d\xi\,dv&=\iint_{\mathbb{R}\times\mathbb{R}}|\xi+\eta|^{2\theta}|\hat{f}(\xi)|^2|\hat{G}(\eta)|^2\,d\xi\,d\eta\\
        &\leq 2u \int_\mathbb{R}\left(|\xi|+\frac{k}{2 u}\right)^{2\theta}|\hat{f}(\xi)|^2\,d\xi
    \end{align*}
    if $k\geq \theta$. Choosing $k$ as $k=\lceil\theta\rceil$ and $u$ as $u=\frac{2\lceil\theta\rceil^2}{\pi}$, we have
    \begin{align}
         2u \int_\mathbb{R}\left(|\xi|+\frac{k}{2 u}\right)^{2\theta}|\hat{f}(\xi)|^2\,d\xi\leq 2eu\int_\mathbb{R}(1+|\xi|^{2\theta})|\hat{f}(\xi)|^2\,d\xi.\label{eq:freq-weight-upper}
    \end{align}
    Combining the bounds (\ref{eq:gv-mixed-space-upper}), (\ref{eq:gv-window-weight-lower}), (\ref{eq:gv-weighted-L2-averaging}), and (\ref{eq:freq-weight-upper}), we get
    \begin{align*}
        &\sum_{\lambda\in\Lambda}|f(\lambda)|^2\max\left\{\tilde{S}^{-1}(|\lambda|)^\frac{2\theta-1}{s},1\right\}\\
        &\leq\frac{12}{c}\int_\mathbb{R}|f(x)|^2\max\left\{\tilde{S}^{-1}(2|x|)^\frac{2\theta-1}{s},1\right\}\,dx+\frac{12}{c}\tilde{S}^{-1}\left(\frac{48}{\pi}\theta^2\right)^\frac{2\theta-1}{s}\int_\mathbb{R}|f(x)|^2\,dx\\
        &+\frac{8e}{c}\int_\mathbb{R}(1+|\xi|^{2\theta})|\hat{f}(\xi)|^2\,d\xi.
    \end{align*}
    since $\lceil\theta\rceil\leq2\theta$. We choose $\theta$ such that $2\frac{\theta}{s}=n\in\mathbb{N}$ and multiply by $\frac{A_n}{n!}$ where $A_n=\frac{d^n}{dz^n}\tilde{S}^2(0)$, this gives us the following inequality
    \begin{align*}
        \frac{A_n}{n!}\sum_{\lambda\in\Lambda}|f(\lambda)|^2
        \max\left\{{\tilde{S}^{-1}(|\lambda|)},\,1\right\}^{n-\frac{1}{s}} 
        &\leq \frac{24e}{c}\,\frac{A_n}{n!}\Bigg(
            \int_{\mathbb{R}} |f(x)|^2
            \max\{{\tilde{S}^{-1}(2|x|)},\,1\}^n\,dx \\
        & + \frac{1}{2e}\,{\tilde{S}^{-1}\!\left(\tfrac{12}{\pi}(ns)^2\right)}^{n}
            \int_{\mathbb{R}} |f(x)|^2\,dx
            + \int_{\mathbb{R}} (1+|\xi|^{ns})|\hat{f}(\xi)|^2\,d\xi
        \Bigg).
    \end{align*}

    By the monotone convergence theorem, we have
    \begin{align*}
        &\sum_{\lambda\in\Lambda}\frac{\max\{|\lambda|^2,1\}|f(\lambda)|^2}{\max\{S^{-1}(|\lambda|),1\}}\\&\leq\frac{24e}{c}\left(\int_\mathbb{R}|f(x)|^2(|x|^2+C_{\tilde{S},s}(12\pi^{-1}s^2))\,dx+\int_\mathbb{R}(S^2(1)+S^2(|\xi|))|\hat{f}(\xi)|^2\,d\xi\right)\\
        &\leq \frac{24e}{c}(C_{\tilde{S},s}(12\pi^{-1}s^2)+S^2(1))\left(\int_\mathbb{R}(1+|x|^2)|f(x)|^2\,dx+\int_\mathbb{R}(1+S^2(|\xi|)^2)|\hat{f}(\xi)|^2\,d\xi\right)
        \end{align*}
    with $C_{\tilde{S},s}$ defined as in (\ref{eq:CStilde}). By Lemma~\ref{lem:endpoint-trace}, we have
    \begin{align*}
        \sum_{\mu\in M}c\max\{S(|\mu|),1\}|\hat{f}(\mu)|^2&\leq2\left(\int_\mathbb{R}x^2|f(x)|^2\,dx+c^2\int_\mathbb{R}\max\{S^2(|\xi|),1\}|\hat{f}(\xi)|^2\,d\xi\right).
    \end{align*}
    Combining these bounds gives us
    \begin{align*}
        &\sum_{\lambda\in\Lambda}\frac{(1+|\lambda|^2)|f(\lambda)|^2}{\max\{S^{-1}(|\lambda|),1\}}+\sum_{\mu\in M}\max\{S(|\mu|^s),1\}|\hat{f}(\mu)|^2\\
        &\leq C\left(\int_\mathbb{R}(1+x^2)|f(x)|^2\,dx+\int_\mathbb{R}(1+S^2(|\xi|))|\hat{f}(\xi)|^2\,d\xi\right)
    \end{align*}
    for some $C>0$.
\end{proof}
The proof of Theorem~\ref{thm:main-sampling-inequality} is now quite simple.
\begin{proof}[Proof of Theorem~\ref{thm:main-sampling-inequality}]
    If $S$ is a polynomial, then $S(t)\sim t^{p-1}$ and $S^{-1}(t)\sim t^{q-1}$ as $t\to\infty$ for some $\frac{1}{p}+\frac{1}{q}=1$. By Theorem 2 in \cite{kulikov2025}, we get the desired bound.
    
    Otherwise, by Lemma~\ref{lem:lower-frame} we get
    the bound
        \begin{align*}
    A\|f\|_{H(S)}^2
    \leq \sum_{j\in\mathbb{Z}}(\lambda_{j+1}-\lambda_{j-1})(1+|\lambda_j|^2)|f(\lambda_j)|^2 + \sum_{j\in\mathbb{Z}}(\mu_{j+1}-\mu_{j-1})(1+ S^2(|\mu_j|))|\hat{f}(\mu_j)|^2
    \end{align*}   
    for some $A>0$. 
    We know that 
    \[
    \lambda_{j+1}-\lambda_{j-1}\leq\frac{2\overline{\alpha}}{\max\{S^{-1}(\min\{|\lambda_{j-1}|,|\lambda_j|\}),K_{\overline{\alpha}, S}\}}\leq\frac{2C_1\overline{\alpha}}{\max\{S^{-1}(|\lambda_j|),K_{\overline{\alpha}, S}\}}
    \]
    for some $C_1>0$, since $S^{-1}(|\lambda_j|)\sim S^{-1}(|\lambda_{j+1}|)$ as $j\to\pm\infty$ by Lemma~\ref{lem:f-shift-asymptotics}. Consequently, 
     \[
    \lambda_{j+1}-\lambda_{j-1}\leq\frac{2C_1\overline{\alpha}}{\max\{S^{-1}(|\lambda_{j-1}|),K_{\overline{\alpha}, S}\}}\leq2 C_1\frac{\overline{\alpha}}{\underline{\alpha}}(\lambda_{j+1}-\lambda_j).
    \]
    By the same reasoning, there exists a $C_2>0$ such that 
    \[
    \mu_{j+1}-\mu_j\leq 2C_2 \frac{\overline{\alpha}}{\underline{\alpha}}(\mu_{j+1}-\mu_j).
    \]
    Thus letting $\tilde{A}:= \frac{A\underline{\alpha}}{2\max\{C_1,C_2\} \overline{\alpha}}$, we have
    \begin{align*}
    \tilde{A}\|f\|_{H(S)}^2
    \leq \sum_{j\in\mathbb{Z}}(\lambda_{j+1}-\lambda_j)(1+|\lambda_j|^2)|f(\lambda_j)|^2 + \sum_{j\in\mathbb{Z}}(\mu_{j+1}-\mu_j)(1+ S^2(|\mu_j|))|\hat{f}(\mu_j)|^2.
    \end{align*} 
    By Lemma~\ref{lem:upper-frame}, we know that
    \begin{align*}
        &\sum_{\lambda\in\Lambda}\frac{(1+|\lambda|^2)|f(\lambda)|^2}{\max\{S^{-1}(|\lambda|),1\}}+\sum_{\mu\in M}\max\{S(|\mu|),1\}|\hat{f}(\mu)|^2\\
        &\leq B\left(\int_\mathbb{R}(1+x^2)|f(x)|^2\,dx+\int_\mathbb{R}(1+S^2(|\xi|))|\hat{f}(\xi)|^2\,d\xi\right)
    \end{align*}
    for some $B>0$. Since 
    \[
    \max\{ S^{-1}(|\lambda_j|),K\}(\lambda_{j+1}-\lambda_j), \max\{ S(|\mu_j|),K\}(\mu_{j+1}-\mu_j)\leq \frac{1}{2},
    \]
    we know that 
    \begin{align*}
        (\lambda_{j+1}-\lambda_j)(1+|\lambda_j|^2)|f(\lambda_j)|^2& \leq\frac{1}{2\min\{K,1\}}\frac{(1+|\lambda_j|^2)|f(\lambda_j)|^2}{\max\{S^{-1}(|\lambda_j|),1\}}\quad\text{and}\\
        (\mu_{j+1}-\mu_j)(1+S^2(|\mu_j|))|\hat{f}(\mu_j)|^2&\leq \frac{1}{\min\{K,1\}}\frac{\max\{S^2(|\mu_j|),1\}|\hat{f}(\mu_j)|^2}{\max\{S(|\mu_j|),1\}}.
    \end{align*}
    Thus, we have 
    \begin{align*}
        &\sum_{j\in\mathbb{Z}}(\lambda_{j+1}-\lambda_j)(1+\lambda_j^2)|f(\lambda_j)|^2+\sum_{j\in\mathbb{Z}}(\mu_{j+1}-\mu_j)(1+S^2(|\mu_j|)|\hat{f}(\mu_j)|^2\\
        &\leq \frac{1}{\min\{K,1\}}\left(\sum_{\lambda\in\Lambda}\frac{(1+|\lambda|^2)|f(\lambda)|^2}{\max\{S^{-1}(|\lambda|),1\}}+\sum_{\mu\in M}S(|\mu|)|\hat{f}(\mu)|^2\right).
    \end{align*}
    Letting $C:=\max\left\{\tilde{A}^{-1},\frac{B}{\min\{K,1\}}\right\}$, we get the desired frame bound.
\end{proof}

\section*{Acknowledgments}
I want to thank Kristian Seip for introducing me to the fascinating topic of Fourier uniqueness and for providing valuable feedback on the exposition. I am grateful to Aleksei Kulikov for suggesting that Lemma~\ref{lem:fractional-poincare-wirtinger} could be applied to more asymmetric pairs of sequences than those considered in an early draft of this paper. I also thank him for suggesting condition (ii) in Definition~\ref{def:widely-admissible} and for
 conjecturing Proposition~\ref{prop:widely-admissible-dominates}. I am also grateful to Mikhail Sodin and Denis Zelent for helpful comments.
\printbibliography
\end{document}